\documentclass[11pt,leqno]{amsart}

\usepackage{fullpage}
\usepackage{amssymb}
\usepackage{amsmath}
\usepackage{amsxtra}
\usepackage{amscd}
\usepackage{graphicx}
\usepackage{amsfonts}
\usepackage{pb-diagram}

\usepackage{float}
\usepackage{enumerate}
\usepackage{dsfont}
\usepackage{multirow}
\usepackage{color}
\usepackage{rotating}
\usepackage{url}
\usepackage{hyperref}
\usepackage[utf8]{inputenc}
\usepackage[all]{xy}
\usepackage{tikz}
\usetikzlibrary{arrows, decorations.markings}
\usepackage[british,UKenglish,USenglish,english,american]{babel}
\usepackage{bicaption}

\makeatletter
\let\@wraptoccontribs\wraptoccontribs
\makeatother
\definecolor{darkgreen}{rgb}{0,0.4,0.1}
\hypersetup{
   colorlinks=true,       	
   linkcolor=red,          
   citecolor=darkgreen,    
   filecolor=magenta,      
   urlcolor=blue			
}

\def\leq{\leqslant}
\def\geq{\geqslant}

\numberwithin{equation}{section}
\newtheorem{thm}{Theorem}

\newtheorem{cor}{Corollary}
\newtheorem{prop}{Proposition}

\newtheorem{exmp}{Example}

\newtheorem{rem}{Remark}

\def\Z{{\mathbb Z}}

\def\N{{\mathbb N}}

\def\YH{{\rm Y}_{d,n}}

\begin{document}

\title[New skein invariants of links]
  {New skein invariants of links}

\author{Louis H.~Kauffman}
\address{Department of Mathematics, Statistics and Computer Science,
University of Illinois at Chicago,
851 South Morgan St., Chicago IL 60607-7045,  U.S.A.}
\email{kauffman@math.uic.edu}
\urladdr{http://www.}

\author{Sofia~Lambropoulou}
\address{Departament of Mathematics,
National Technical University of Athens,
Zografou campus, GR-157 80 Athens, Greece.}
\email{sofia@math.ntua.gr}
\urladdr{http://www.math.ntua.gr/~sofia}



\thanks{ The authors acknowledge with pleasure valuable conversations with Akinori John Nakaema and Aristides Kontogeorgis about the structure of the proof of Proposition~\ref{orderxings}. 
 This research  has been co-financed by the European Union (European Social Fund - ESF) and Greek national funds through the Operational Program ``Education and Lifelong Learning" of the National Strategic Reference Framework (NSRF) - Research Funding Program: THALES: Reinforcement of the interdisciplinary and/or inter-institutional research and innovation, MIS: 380154. L. K. is pleased to thank the Simons Foundation, USA, for partial support for this research under his Collaboration Grant for Mathematicians (2016 - 2021), Grant number 426075. }


\date{}

\begin{abstract}
We introduce new skein invariants of  links based on a procedure where we first apply the skein relation only to crossings of distinct components, so as to produce collections of unlinked knots. We then evaluate the resulting knots using a given invariant.  A skein invariant can be computed on each link solely by the use of skein relations and a set of initial conditions.  The new procedure, remarkably, leads to  generalizations of the known skein invariants. We make skein  invariants  of classical links, $H[R]$, $K[Q]$  and $D[T]$, based on the invariants of knots, $R$, $Q$  and $T$, denoting the regular isotopy version of the Homflypt polynomial, the Kauffman polynomial and the Dubrovnik polynomial. 
We provide skein theoretic proofs of the well-definedness of these invariants. 
These invariants  are also reformulated into summations of the generating invariants ($R$, $Q$, $T$) on sublinks of a  
 given link $L$, obtained by partitioning $L$ into collections of sublinks. 
 \end{abstract}

\maketitle

\setcounter{tocdepth}{1}
\begin{center}
\begin{minipage}{10cm}
{\scriptsize\tableofcontents}
\end{minipage}
\end{center}

\section*{Introduction} \label{s:intro}

Skein theory was introduced by John Horton Conway in his remarkable paper \cite{Conway} and in numerous conversations and lectures that Conway gave in the wake of
publishing this paper. He not only discovered a remarkable normalized recursive method to compute the classical Alexander polynomial, but he formulated a generalized invariant
of knots and links that is called {\it skein theory}. In the general form of this theory  an ambient isotopy class of an oriented knot or link diagram $K$ is represented by placing brackets around it
as in $\{K\}$ and is described by Conway as a three-dimensional ``room" containing the knot or link $K.$ If $K_{+}, K_{-}, K_{0}$ denote three diagrams that differ at the site of one crossing, with a positive crossing at $K_{+},$ a negative crossing at $K_{-}$ and a smoothed crossing at $K_{0}$ (this relationship is described in detail below), then Conway writes 
$$
\{ K_{+} \} = \{ K_{-} \} \oplus \{ K_{0} \}
$$ 
where it is understood that $\oplus$ is a non-associative operation that brings together two diagrams (in the indicated isotopy types) and makes them equivalent to a single diagrammatic isotopy type as indicated by the equation. Similarly one writes
$$\{ K_{-} \} = \{ K_{+} \} \ominus \{ K_{0} \}.$$ A knot or link can be decomposed into an expression in these operations acting on rooms filled only with unknots and unlinks.
Such an expression is called a skein decomposition. We call two links $L$ and $L'$ {\it skein equivalent} if they have identical skein decompositions. It is an open problem to this day to determine the skein equivalence class of a given link.

Conway gave examples of specific representations of the skein. The most famous of these is the representation, called the Conway-Alexander polynomial, that behaves as follows: $$\nabla(A \oplus B) = \nabla(A) + z \nabla(B)$$ and $\nabla$ is equal to $1$ on the unknot and $0$ on unlinks of more than one component. Later, in the wake of the 
discovery of the Jones polynomial other representations of the skein were discovered, including the famous Homflypt polynomial and the Kauffman polynomial defined for a related 
unoriented skein that had not been defined by Conway.

\smallbreak

In this paper we introduce new {\it link skein theories}  in the form of specific procedures for handling the skein equivalence. We only allow 
$$
\{ L_{+} \} = \{ L_{-} \} \oplus \{ L_{0} \}$$or $$\{ L_{-} \} = \{ L_{+} \} \ominus \{ L_{0} \}
$$ 
when the arcs that are switched are on different link components. This means that the ``bottom" of our new skein consists in rooms filled with disjoint unions of possibly non-trivial knots. We then ask {\it what is the link skein equivalence class of a link}?
We do not know the answer in general, but we do generalize the Homflypt and Kauffman polyomials to this link skein theory and show that our new invariants are indeed stronger than the
original skein polynomials.

A study such as ours and the general formulation of new skein theories is of interest to knot theorists and to topologists generally. The idea of skein decompositon and equivalence  can be carried beyond classical knot theory (This will be the subject of other papers that we write.) and the specific constructions in the present paper will be a guide to further exploration. The reader may want to compare these remarks with the constructions of Seongjeong Kim  in his paper that uses the Conway algebra \cite{Kim}.

 More precisely, in this paper we introduce generalized skein invariants of links, $H[R]$, $K[Q]$  and $D[T]$, based on the regular isotopy version of the 2-variable Jones or Homflypt polynomial, the Kauffman polynomial and the Dubrovnik polynomial, respectively. By a {\it skein invariant} we mean that it can be computed on each link solely by the use of certain specific linear skein relations and a set of initial conditions.  The invariant $H[R]$ was originally discovered via Yokonuma--Hecke algebra traces and was called $\Theta$ \cite{chjukala}. We give here, for the first time, a direct skein theoretic treatment of $\Theta$ (as $H[R]$). It was further  discovered by Lickorish that  $\Theta$ on a given link can be expressed by a summation over products of Homflypt evaluations of sublinks \cite[Appendix]{chjukala}. We give   new generalizations of the Kauffman and Dubrovnik polynomials by skein theory and by  Lickorish-type formulae (see \S~\ref{closedforms}). This paper, formulating the invariants in pure skein theory and combinatorics, provides a foundation for this research and, we hope, a springboard for further interaction with knot algebras.

In making a skein theory for an invariant of knots and links one uses a linear  formula that involves a crossing in a diagram, the switched version of that crossing and other replacements for the crossing. The key idea in our skein theory for the invariants in this paper is to order skein operations so that crossings between different components are given precedence in the order of switching.  These constructions alter the philosophy of classical skein-theoretic techniques, whereby mixed as well as self-crossings in a link diagram would get indiscriminantly switched.  In the present approach one first unlinks all components using the skein relation of  a known skein invariant and then one evaluates  that skein invariant on  unions of unlinked knots, at the same time introducing a new variable. Such a skein decomposition is indicated for a simple 2-component link in Figure~\ref{skeindecomp}, where $U$ indicates the unknot and $3_1$ a trefoil knot.

\smallbreak
\begin{figure}[!ht]
\begin{center}
\includegraphics[width=6cm]{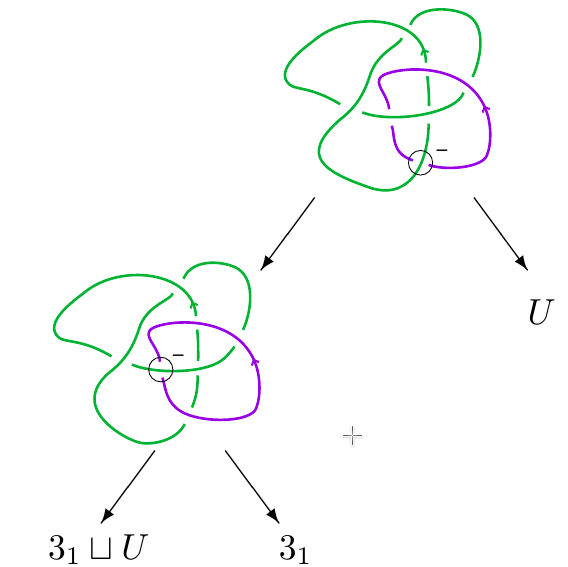}
\caption{A skein decomposition to the component knots}
\label{skeindecomp}
\end{center}
\end{figure}

The invariant $H[R]$ uses the Homflypt skein relation  for crossings between different components and the Homflypt polynomial for evaluating on  knots. More precisely, we abstract the skein relation of the  regular isotopy version of the Homflypt polynomial,   denoted  $R$, and use it as the basis of a new skein algorithm comprising two computational levels : on the first level we only apply the skein relation between {\it mixed crossings}, that is, crossings of different components, so as to produce {\it unions of unlinked knots}. On the second level we evaluate the invariant $H[R]$ on unions of unlinked knots, by applying a new rule,  which is based on the evaluation of  $R$ on unions of unlinked knots and which introduces a new variable. The invariant $R$ is evaluated on unions of unlinked knots through its evaluation  on individual knots. 
 Thus the invariant $H[R]$ generalizes the Homflypt polynomial. 
  This method generalizes the skein invariants Homflypt and Kauffman (Dubrovnik) to new invariants of links.

We note that there are few known skein invariants  in the literature for classical knots and links. Skein invariants include: the Alexander--Conway polynomial \cite{al,co}, the Jones polynomial \cite{jo1}, and the Homflypt polynomial \cite{oc,jo2, LM,HOMFLY,PT}, which specializes to both the Alexander--Conway and the Jones polynomial; there is also the bracket polynomial \cite{kau2}, the Brandt--Lickorish--Millett--Ho polynomial \cite{BLM}, the Dubrovnik polynomial and the Kauffman polynomial \cite{kau4}, which specializes to both the bracket and the Brandt--Lickorish--Millett--Ho polynomial. Finally, we have the  Juyumaya--Lambropoulou  family of invariants $\Delta_{d,D}$ \cite{jula2}, and the analogous Chlouveraki--Juyumaya--Karvounis--Lambropoulou family of invariants  $\Theta_d(q, \lambda_d)$ and their generalization $\Theta(q, \lambda, E)$ \cite{chjukala}, which specializes to the Homflypt polynomial. One avenue for further research is to extend the approach of the present paper to the area of skein modules and invariants of links in three-manifolds, and invariants of three-manifolds (see for example \cite{Tu,P,HK,La1,La2,GM,DL3,chpa2,fljula}). 

	\smallbreak
	
Let $\mathcal{L}$ denote the set of classical oriented link diagrams. Let also $L_+$ be an oriented diagram with a positive crossing specified and let $L_-$ be the same diagram but with that crossing switched. Let also $L_0$ indicate the same  diagram but with the smoothing which is compatible with the orientations of the emanating arcs in place of the crossing, see (\ref{triple}). The diagrams $L_+, L_-, L_0$  comprise a so-called {\it oriented Conway triple}. 

\begin{equation}\label{triple}
\raisebox{-.1cm}{\begin{tikzpicture}[scale=.2]
\draw [line width=0.35mm]  (-1,-1)-- (-0.22,-0.22);
\draw  [line width=0.35mm](-1,1)--(0,0);
\draw  [line width=0.35mm] (0.22,0.22) -- (1,1)[->];
\draw [line width=0.35mm]   (0,0) -- +(1,-1)[->];
\end{tikzpicture}}
\qquad \qquad
 \raisebox{-.1cm}{\begin{tikzpicture}[scale=.2]
\draw  [line width=0.35mm] (-1,-1)-- (0,0) ;
\draw [line width=0.35mm] (-1,1)--(-0.22,0.22);
\draw [line width=0.35mm] (0,0) -- (1,1)[->];
\draw [line width=0.35mm]   (0.22,-0.22) -- +(.8,-.8)[->];
\end{tikzpicture}} 
\qquad \qquad
\raisebox{-.1cm}{\begin{tikzpicture}[scale=.2, mydeco/.style = {decoration = {markings, 
                                                       mark = at position #1 with {\arrow{>}}}
                                       }]
\draw [line width=0.35mm, postaction = {mydeco=.6 ,decorate}] plot [smooth, tension=2] coordinates { (-1,.8) (0, 0.5) (1,.8)};
\draw [line width=0.35mm, postaction = {mydeco=.6 ,decorate}] plot [smooth, tension=2] coordinates { (-1,-.8) (0, -0.5) (1,-.8)};
\end{tikzpicture}}
\end{equation}
\[ 
L_+ \qquad \qquad L_- \qquad  \qquad L_0
\]

\noindent We then prove the following:

\begin{thm} \label{hofr}
Let $H (z,a)$ denote the regular isotopy version of  the Homflypt polynomial and let $R (w,a)$ denote  the same invariant but with a different indeterminate $w$ in place of $z$.  Then there exists a unique regular isotopy invariant of classical oriented links $H[R]: \mathcal{L} \rightarrow \Z[z, w, a^{\pm 1}, E^{\pm 1}]$, where $z, \, w , \, a$ and $E$ are indeterminates, defined by the following rules:
\begin{enumerate}
\item On crossings involving different components the following mixed skein relation holds for an oriented Conway triple:
$$
H[R](\includegraphics[scale=0.5]{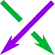}) - H[R](\includegraphics[scale=0.5]{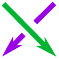}) = z \, H[R](\includegraphics[scale=0.5]{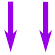}),
$$
\item For a union of $r$ unlinked knots, ${\mathcal K}^r := \sqcup_{i=1}^r K_i$, with $r \geq 1$, it holds that:
$$
H[R]({\mathcal K}^r) =  E^{1-r} \, R({\mathcal K}^r).
$$
\end{enumerate}
\end{thm}

We recall that the invariant $R(w,a)$ is determined by the following rules:
\begin{enumerate}
\item[(R1)] For $L_+$, $L_-$, $L_0$ an oriented Conway triple, the following skein relation holds:
$$
R(L_+) - R(L_-) = w \, R(L_0),
$$
\item[(R2)] The indeterminate $a$ is the positive curl value for $R$: 
$$
R ( \raisebox{-.1cm}{
\begin{tikzpicture}[scale=.2]
\draw [line width=0.35mm]  (-.7,-.7)-- (-0.22,-0.22);
\draw  [line width=0.35mm ](-.7,.7)--(0,0);
\draw  [line width=0.35mm] (0.22,0.22) -- (.7,.7)[->];
\draw [line width=0.35mm]   (0,0) -- +(.7,-.7)[->];
 \draw [line width=0.35mm] plot [smooth, tension=2] coordinates { (-.7,.7) (0,1.3) (.7,.7)};
\end{tikzpicture}}
\ ) = a \, R (  
\raisebox{.06cm}{
\begin{tikzpicture}[scale=.2, mydeco/.style = {decoration = {markings, 
                                                       mark = at position #1 with {\arrow{>}}}
                                       }]
 \draw [line width=0.35mm, postaction = {mydeco=.6 ,decorate}] plot [smooth, tension=2] coordinates {(0,0) (1,.2) (2,0)};
 \end{tikzpicture}}\ ) 
\quad \mbox{and} \quad
 R (
\raisebox{-.1cm}{\begin{tikzpicture}[scale=.2]
\draw  [line width=0.35mm] (-.7,-.7)-- (0,0) ;
\draw [line width=0.35mm] (-.7,.7)--(-0.22,0.22);
\draw [line width=0.35mm] (0,0) -- (.7,.7)[->];
\draw [line width=0.35mm]   (0.22,-0.22) -- +(.6,-.6)[->];
 \draw [line width=0.35mm] plot [smooth, tension=2] coordinates { (-.7,.7) (0,1.3) (.7,.7)};
\end{tikzpicture}}
) = a^{-1} \, R (\raisebox{.06cm}{
\begin{tikzpicture}[scale=.2, mydeco/.style = {decoration = {markings, 
                                                       mark = at position #1 with {\arrow{>}}}
                                       }]
 \draw [line width=0.35mm, postaction = {mydeco=.6 ,decorate}] plot [smooth, tension=2] coordinates {(0,0) (1,.2) (2,0)};
 \end{tikzpicture}}\ ),
$$
\item[(R3)] On the standard unknot:
$$
R(\bigcirc) = 1. 
$$
Hence, $H[R](\bigcirc) = 1$.  We also recall that the above defining rules imply  the following:
\item[(R4)] For a diagram of the unknot, $U$, $R$ is evaluated by taking:
$$
R(U) = a^{wr(U)},
$$ 
where $wr(U)$ denotes the writhe of $U$ --instead of 1 that is the case in the ambient isotopy category. 
\item[(R5)] $R$ being the Homflypt polynomial, it is multiplicative on a union of unlinked knots, ${\mathcal K}^r := \sqcup_{i=1}^r K_i$. Namely, for $\eta := \frac{a - a^{-1}}{w}$ we have:  
$$
R({\mathcal K}^r) = \eta^{r-1} \Pi_{i=1}^r  R(K_i). 
$$
\end{enumerate}

	Furthermore, as we show in \S~\ref{secphi}, Theorem~\ref{theta_linking_P}, for an oriented link  $L$ on $n$ components, we have:
\begin{equation} \label{lickh} 
H[R](L) = \left( \frac{z}{w} \right)^{n-1} \sum_{k=1}^n \eta^{k-1} \widehat{E}_k \sum_\pi R(\pi L)
\end{equation} 
where the second summation is over all partitions $\pi$ of the components of $L$ into $k$ (unordered) subsets and $R(\pi L)$ denotes the product of the Homflypt polynomials of the $k$ sublinks of $L$ defined by $\pi$. Also, $\widehat{E}_k = (\widehat{E}^{-1} - 1)(\widehat{E}^{-1} - 2) \cdots (\widehat{E}^{-1} - k + 1)$, with $\widehat{E} = E \frac{z}{w}$ and $\widehat{E}_1 =1$. 
 The reader should note that the formula (\ref{lickh}) above (the right hand side) is, by its very definition, a regular isotopy invariant of the link $L$. This follows from the regular isotopy invariance of $R$ and the well-definedness of summing over all partitions of the link $L$ into $k$ parts. In fact the summations $I_{k}(L) = \sum_\pi R(\pi L)$, where $\pi$ runs over all partitions of $L$ into $k$ parts, are each regular isotopy invariants of $L$. What is remarkable here is that these all assemble into the new invariant $H[R](L)$ with its striking two-level skein relation. We see from this combinatorial formula how  the extra strength of $H[R](L)$ comes from its ability to detect non-triviality  of certain sublinks of the link $L$. The fact that $H[R]$ is in fact stronger than $R$ is documented in \S~\ref{sectheta}, based on \cite{chjukala}, for the invariant $\Theta$, and in \S~\ref{secphi}.

Assuming Theorem~\ref{hofr} one can  compute $H[R]$ on any given oriented link diagram $L$ by applying the following procedure:  the skein rule (1) of Theorem~\ref{hofr} can be used to give an evaluation of $H[R](L_+)$ in terms of $H[R](L_-)$ and $H[R](L_0)$ or of $H[R](L_-)$ in terms of $H[R](L_+)$ and $H[R](L_0)$. We choose to switch mixed crossings so that the switched diagram is more unlinked than before. Applying this principle recursively we obtain a sum with polynomial coefficients and evaluations of $H[R]$ on  unions of unlinked knots. These knots are formed  by the mergings of components caused by the smoothings in the  skein relation (1). To evaluate $H[R]$ on a given union of unlinked knots we then use the invariant $R$ according to rule (2) of Theorem~\ref{hofr}. Note that the appearance of the indeterminate $E$ in  rule (2) for $H[R]$ is the critical difference between  $H[R]$ and $R$.
 Finally, formula	(R5) above allows evaluations  of the invariant $R$  on individual knotted components   and knowledge of $R$ provides the basis for this. 

  For proving Theorem~\ref{hofr} one must prove that the resulting evaluation is independent of the choices made and invariant under regular isotopy moves. A good guide for this is the skein-theoretic proof of Lickorish--Millett of the well-definedness of the Homflypt polynomial \cite{LM}, so we will be following in principle \cite{LM} with the necessary  adaptations and modifications, taking for granted the well-definedness of $R$. 	The difference here lies in modifying the original skein method which bottoms out on unlinks --since self-crossings are not distinguished from mixed crossings-- to the present context, where the evaluations bottom out on unions of unlinked 	knots.   Then the main difficulty lies in proving independence of the computations under the sequence of mixed crossing switches (see Proposition~\ref{orderxings}). Take for example the case where two mixed crossings need  to be switched, which are shared between the same components, as abstracted in Figure~\ref{ijjifig}. Then, switching the first  causes  the second one to become a self-crossing in the smoothed diagram, so we cannot operate on it. Changing now the order of switches results in different diagrams. In order to compare the polynomials before and after one needs first to reach down to the unlinked components of the diagrams and then apply the skein relation of the invariant $R$ on the self-crossings in question. The complexity of the argument depends on the number of mixed crossings shared by the same components that need to be switched.	
	
\smallbreak
\begin{figure}[!ht]
\begin{center}
\includegraphics[width=14cm]{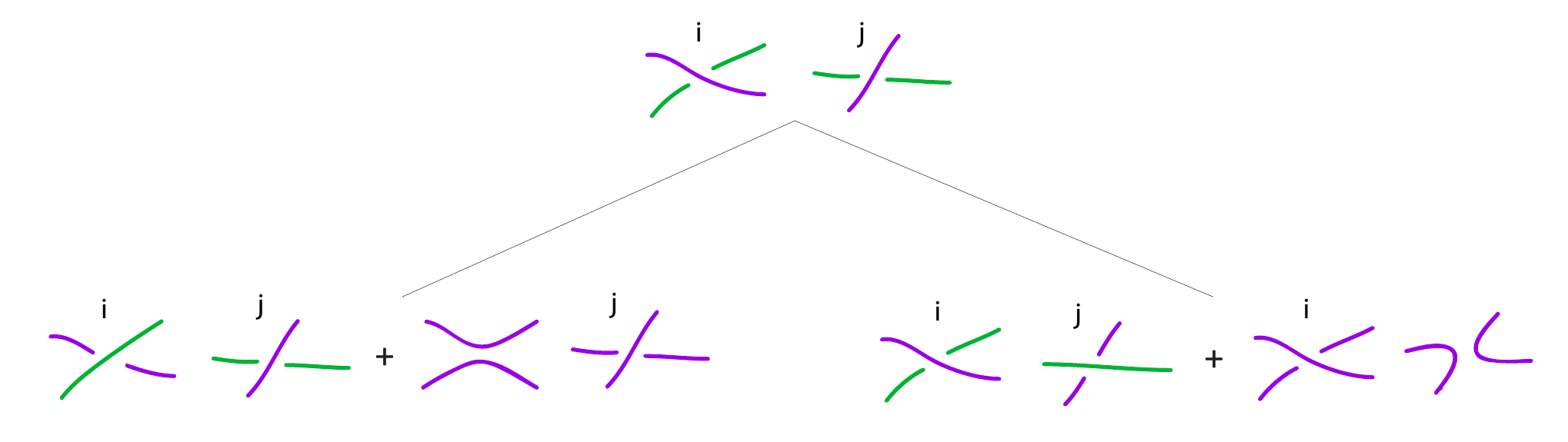}
\caption{Changing the sequence of switching mixed crossings}
\label{ijjifig}
\end{center}
\end{figure}
	
Our motivation for the above generalization  $H[R]$ of the invariant $R$ is the following: 
In \cite{chjukala} the ambient isotopy invariants $\Theta_d(q,\lambda_d)$ for classical links, which were originally derived from a Markov trace on the Yokonuma--Hecke algebras \cite{jula2}, are recovered via a skein relation of the Homflypt polynomial, $P$, that can only apply to mixed crossings of a link. 
  The invariants $\Theta_d$  are then compared to $P$ and are shown  to be  {\it distinct} from $P$ {\it on links}.  The invariants  $\Theta_d$ are also distinct from the  Kauffman polynomial, since they are  {\it topologically equivalent} to $P$ {\it on knots} \cite{chmjakala, chjukala}. Moreover, in  \cite{chjukala}  the family of invariants $\{\Theta_d\}_{d\in\N}$, which includes $P$ for $d=1$, is generalized to a new 3-variable skein link invariant $\Theta(q,\lambda,E)$, which specializes to each $\Theta_d$ for $E=1/d$ and which is stronger than $P$, as proved in \cite{chjukala} and detailed in \S~\ref{sectheta}. The invariant $H[H]$ is the regular isotopy counterpart of  $\Theta$.  Theorem~\ref{hofr} provides a new self-contained skein theoretic proof of the existence of the invariant $\Theta$ of \cite{chjukala}, which is missing in the literature. 
		In \cite[Appendix B]{chjukala} W.B.R. Lickorish provides a closed combinatorial formula (Eq.~\ref{lickorish}) for the definition of the invariant $\Theta$, showing that it is a mixture  of Homflypt polynomials and linking numbers of sublinks of a given link.  A succinct exposition of the above can be found in \cite{kala}.	These  constructions opened the way to new research directions, cf. \cite{jula1,jula2,jula3,jula4,jula5,gou,gojukola,ChPou1,gojukolaf,ChPou2,chla,chpa1,chmjakala,japa,chjukala,goula1,ju2,goula2,pawa,AJ1,AJ2,AJ3,chpa2,fljula}.

	Theorem~\ref{hofr} implies that, if we wish, we could specialize the $z$, the $w$, the $a$ and the $E$ in any way we wish. For example, if $a=1$ then $R$ specializes to the Alexander--Conway polynomial \cite{al,co}. If $w= \sqrt{a} - 1/\sqrt{a}$ then $R$ becomes the unnormalized Jones polynomial \cite{jo1}. In each case $H[R]$ can be regarded as a generalization of that polynomial. Furthermore, we denote by $H[H]$ the invariant $H[R]$ in the case where $w=z$.  The invariant $H[H]$ still generalizes $H$ to a new 3-variable invariant for {\it links}.  Indeed, in this case $H[H]$ is the regular isotopy version of the new 3-variable link invariant $\Theta (q, \lambda, E)$ \cite{chjukala}. Our 4-variable invariant $H[R]$ is in fact topologically equivalent to the 3-variable invariant $H[H]$. This fact  was observed by K.~Karvounis \cite{ka2}  by modifying our previous version of  the Lickorish-type combinatorial formula (\ref{lickh}).  
	 Thus we now know that all the invariants in this paper are equivalent to 3-variable invariants, even though we have formulated them using four variables. In  \cite{kaula2} we use the 3-variable formulation exclusively. The 4-variable formulation is, nevertheless, useful for separating the two types of skein operations. Namely, switching crossings between different components and evaluating on knots.

\smallbreak

We now consider the class $\mathcal{L}^u$ of unoriented link diagrams. For any crossing of a diagram of a link in  $\mathcal{L}^u$, if we swing the overcrossing arc counterclockwise it sweeps two regions out of the four. If we join these two regions, this is the $A$-smoothing of the crossing, while joining the other two regions gives rise to the $B$-smoothing. We shall say that a crossing is of {\it positive type} if it produces a horizontal $A$-smoothing and that it is of {\it negative type} if it produces a vertical $A$-smoothing. 
Let now $L_+$ be an unoriented diagram  with a positive type crossing specified and let $L_-$ be the same diagram but with that crossing switched. Let also $L_0$ and $L_{\infty}$ indicate the same  diagram but with the $A$-smoothing and the $B$-smoothing in place of the crossing. See (\ref{quadruple}).  The diagrams $L_+, L_-, L_0, L_{\infty}$  comprise a so-called {\it unoriented Conway quadruple}. 

\begin{equation}\label{quadruple}
 \raisebox{-.1cm}{\begin{tikzpicture}[scale=.2]
\draw [line width=0.35mm]  (-1,-1)-- (-0.22,-0.22);
\draw  [line width=0.35mm ](-1,1)--(0,0);
\draw  [line width=0.35mm] (0.22,0.22) -- (1,1);
\draw [line width=0.35mm]   (0,0) -- +(1,-1);
\end{tikzpicture}} 
\qquad \qquad
 \raisebox{-.1cm}{\begin{tikzpicture}[scale=.2]
\draw  [line width=0.35mm] (-1,-1)-- (0,0) ;
\draw [line width=0.35mm] (-1,1)--(-0.22,0.22);
\draw [line width=0.35mm] (0,0) -- (1,1);
\draw [line width=0.35mm]   (0.22,-0.22) -- +(.8,-.8);
\end{tikzpicture}}  
\qquad \qquad
\raisebox{-.07cm}{\begin{tikzpicture}[scale=.2, mydeco/.style = {decoration = {markings, 
                                                       mark = at position #1 with {\arrow{>}}}
                                       }]
\draw [line width=0.35mm] plot [smooth, tension=2] coordinates { (-1,.8) (0, 0.5) (1,.8)};
\draw [ line width=0.35mm] plot [smooth, tension=2] coordinates { (-1,-.8) (0, -0.5) (1,-.8)};
\end{tikzpicture}} 
\qquad  \qquad
\raisebox{-.1cm}{\begin{tikzpicture}[scale=.2]
 \draw [ line width=0.35mm] plot [smooth, tension=2] coordinates { (-1,-1) (-0.3, 0) (-1,1)};
 \draw [ line width=0.35mm] plot [smooth, tension=2] coordinates { (1,-1) (0.3, 0) (1,1)};
 \end{tikzpicture}} 
\end{equation}
\[ 
L_+ \qquad \qquad L_- \qquad  \qquad L_0 \qquad  \qquad L_{\infty}
\]

 By similar arguments as for Theorem~\ref{hofr} we also prove in this paper the existence of 4-variable generalizations of the regular isotopy versions of the Dubrovnik and the Kauffman polynomials \cite{kau4}. Namely:

\begin{thm} \label{doft}
Let $D (z,a)$ denote the regular isotopy version of the Dubrovnik polynomial and let $T (w,a)$ denote the same invariant but with a different parameter $w$ in place of $z$.  Then there exists a unique regular isotopy invariant of classical unoriented links $D[T]:  \mathcal{L}^u \rightarrow \Z[z, w, a^{\pm 1}, E^{\pm 1}]$, where $z, \, w , \, a$ and $E$ are indeterminates, defined by the following rules:
\begin{enumerate}
\item On crossings involving different components the following skein relation holds for an unoriented Conway quadruple:
$$
D[T] (\includegraphics[scale=0.5]{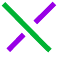}) - D[T] (\includegraphics[scale=0.5]{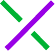}) = z \, \big( D[T] (\includegraphics[scale=0.5]{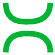}) - D[T] (\includegraphics[scale=0.5]{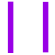}) \big),
$$
\item For a  union of $r$ unlinked knots in $\mathcal{L}^u$, ${\mathcal K}^r := \sqcup_{i=1}^r K_i$, with $r \geq 1$, it holds that:
$$
D[T] ({\mathcal K}^r) =  E^{1-r} \, T ({\mathcal K}^r).
$$
\end{enumerate}
\end{thm}

We recall that the invariant $T(w,a)$ is determined by the following rules:
\begin{enumerate}
\item[(T1)] For $L_+$, $L_-$, $L_0$, $L_{\infty}$ an unoriented Conway quadruple, the following skein relation holds:
$$
T (L_+) - T (L_-) = w \, \big( T (L_0) -  T (L_{\infty}) \big),
$$
\item[(T2)] The indeterminate $a$ is the positive type curl value for $T$: 
$$
T ( \raisebox{-.1cm}{
\begin{tikzpicture}[scale=.2]
\draw [line width=0.35mm]  (-.7,-.7)-- (-0.22,-0.22);
\draw  [line width=0.35mm ](-.7,.7)--(0,0);
\draw  [line width=0.35mm] (0.22,0.22) -- (.7,.7);
\draw [line width=0.35mm]   (0,0) -- +(.7,-.7);
 \draw [line width=0.35mm] plot [smooth, tension=2] coordinates { (-.7,.7) (0,1.3) (.7,.7)};
\end{tikzpicture}}
\ ) = a \, T (  
\raisebox{.06cm}{
\begin{tikzpicture}[scale=.2, mydeco/.style ]
 \draw [line width=0.35mm, postaction = {mydeco=.6 ,decorate}] plot [smooth, tension=2] coordinates {(0,0) (1,.2) (2,0)};
 \end{tikzpicture}}\ ) 
\quad \mbox{and} \quad
 T (
\raisebox{-.1cm}{\begin{tikzpicture}[scale=.2]
\draw  [line width=0.35mm] (-.7,-.7)-- (0,0) ;
\draw [line width=0.35mm] (-.7,.7)--(-0.22,0.22);
\draw [line width=0.35mm] (0,0) -- (.7,.7);
\draw [line width=0.35mm]   (0.22,-0.22) -- +(.6,-.6);
 \draw [line width=0.35mm] plot [smooth, tension=2] coordinates { (-.7,.7) (0,1.3) (.7,.7)};
\end{tikzpicture}}
) = a^{-1} \, T (
\raisebox{.06cm}{
\begin{tikzpicture}[scale=.2, mydeco/.style ]
 \draw [line width=0.35mm, postaction = {mydeco=.6 ,decorate}] plot [smooth, tension=2] coordinates {(0,0) (1,.2) (2,0)};
 \end{tikzpicture}}\ ),
$$
\item[(T3)] On the standard unknot:
$$
T(\bigcirc) = 1. 
$$
So,  $D[T](\bigcirc) = 1$.  We also recall that the above defining rules imply  the following:
\item[(T4)] For a diagram of the unknot, $U$, $T$ is evaluated by taking 
$$
T(U) = a^{wr(U)},
$$ 
\item[(T5)] $T$, being  the  Dubrovnik  polynomial,  it is multiplicative on a union of unlinked knots, ${\mathcal K}^r := \sqcup_{i=1}^r K_i$. Namely, for $\delta := \frac{a - a^{-1}}{w} + 1$ we have:  
$$
T ({\mathcal K}^r) = \delta^{r-1} \Pi_{i=1}^r T (K_i). 
$$
\end{enumerate} 

\noindent Furthermore, as for $H[R]$, a closed combinatorial Lickorish-type formula exists for $D[T]$. Namely, as we prove in \S~\ref{secxi}, Theorem~\ref{thmxi}, on any unoriented link $L$ we have  :
\begin{equation} \label{lickd} 
D[T](L) = (\frac{z}{w})^{n-1}\sum_{k=1}^n \delta^{k-1}\widehat{E_k} \sum_\pi T(\pi L).
\end{equation} 
All terms are analogous to the terms involved in (\ref{lickh}).

The Dubrovnik polynomial, $D$, is related to  the Kauffman polynomial, $K$, via the following translation formula, observed by W.B.R~Lickorish~\cite{kau4}:
\begin{equation} \label{ktod}
D(L)(a,z) = (-1)^{c(L)+1} \, i^{-wr(L)}K(L)(ia,-iz).
\end{equation} 
Here, $c(L)$ denotes the number of components of $L$, $i^2 = -1$, and $wr(L)$ is the {\it writhe} of $L$ for some choice of orientation of $L$, which is defined as the algebraic sum of all crossings of $L$.  The translation formula is independent of the particular choice of orientation for
$L$. Our theory generalizes also the regular isotopy version of the Kauffman polynomial \cite{kau4} through the following:

\begin{thm} \label{kofq}
Let $K (z,a)$ denote the regular isotopy version of the Kauffman polynomial and let $Q (w,a)$ denote the same invariant but with a different parameter $w$ in place of $z$.  Then there exists a unique regular isotopy invariant of classical unoriented links $K[Q]:  \mathcal{L}^u \rightarrow \Z[z, w, a^{\pm 1}, E^{\pm 1}]$, where $z, \, w , \, a$ and $E$ are indeterminates, defined by the following rules:
\begin{enumerate}
\item On crossings involving different components the following skein relation holds for an unoriented Conway quadruple:
$$
K[Q]  (\includegraphics[scale=0.5]{crossing_pos_u.pdf})  + K[Q] (\includegraphics[scale=0.5]{crossing_neg_u.pdf}) = z \, \big( K[Q] (\includegraphics[scale=0.5]{crossing_zero_u.pdf}) +  K[Q] (\includegraphics[scale=0.5]{crossing_infty_u.pdf}),
$$
\item For a union of $r$ unlinked knots in $\mathcal{L}^u$, ${\mathcal K}^r := \sqcup_{i=1}^r K_i$, with $r \geq 1$, it holds that:
$$
K[Q]({\mathcal K}^r) =  E^{1-r} \, Q({\mathcal K}^r).
$$
\end{enumerate}
\end{thm}

We recall that the invariant $Q(w,a)$ is determined by the following rules:
\begin{enumerate}
\item[(Q1)] For $L_+$, $L_-$, $L_0$, $L_{\infty}$ an unoriented Conway quadruple, the following skein relation holds:
$$
Q (L_+) + Q (L_-) = w \, \big( Q (L_0) +  Q (L_{\infty}) \big),
$$
\item[(Q2)] The indeterminate $a$ is the positive type curl value for $Q$: 
$$
Q ( \raisebox{-.1cm}{
\begin{tikzpicture}[scale=.2]
\draw [line width=0.35mm]  (-.7,-.7)-- (-0.22,-0.22);
\draw  [line width=0.35mm ](-.7,.7)--(0,0);
\draw  [line width=0.35mm] (0.22,0.22) -- (.7,.7);
\draw [line width=0.35mm]   (0,0) -- +(.7,-.7);
 \draw [line width=0.35mm] plot [smooth, tension=2] coordinates { (-.7,.7) (0,1.3) (.7,.7)};
\end{tikzpicture}}
\ ) = a \, Q (  
\raisebox{.06cm}{
\begin{tikzpicture}[scale=.2, mydeco/.style ]
 \draw [line width=0.35mm, postaction = {mydeco=.6 ,decorate}] plot [smooth, tension=2] coordinates {(0,0) (1,.2) (2,0)};
 \end{tikzpicture}}\ ) 
\quad \mbox{and} \quad
 Q (
\raisebox{-.1cm}{\begin{tikzpicture}[scale=.2]
\draw  [line width=0.35mm] (-.7,-.7)-- (0,0) ;
\draw [line width=0.35mm] (-.7,.7)--(-0.22,0.22);
\draw [line width=0.35mm] (0,0) -- (.7,.7);
\draw [line width=0.35mm]   (0.22,-0.22) -- +(.6,-.6);
 \draw [line width=0.35mm] plot [smooth, tension=2] coordinates { (-.7,.7) (0,1.3) (.7,.7)};
\end{tikzpicture}}
) = a^{-1} \, Q (
\raisebox{.06cm}{
\begin{tikzpicture}[scale=.2, mydeco/.style ]
 \draw [line width=0.35mm, postaction = {mydeco=.6 ,decorate}] plot [smooth, tension=2] coordinates {(0,0) (1,.2) (2,0)};
 \end{tikzpicture}}\ ),
$$
\item[(Q3)] On the standard unknot:
$$
Q(\bigcirc) = 1. 
$$
So, $K[Q](\bigcirc) = 1$.  We also recall that the above defining rules imply  the following:
\item[(Q4)] For a diagram of the unknot, $U$, $Q$ is evaluated by taking 
$$
Q(U) = a^{wr(U)},
$$ 
\item[(Q5)] $Q$, being  the Kauffman polynomial,  it is multiplicative on a union of unlinked knots, ${\mathcal K}^r := \sqcup_{i=1}^r K_i$. Namely, for $\gamma := \frac{a + a^{-1}}{w} - 1$ we have:  
$$
Q({\mathcal K}^r) = \gamma^{r-1} \Pi_{i=1}^r Q(K_i). 
$$
\end{enumerate} 

\noindent Furthermore, in \S~\ref{secpsi}, Theorem~\ref{thmpsi} we prove that $K[Q]$ on an unoriented link $L$ satisfies the following closed combinatorial Lickorish-type formula:
\begin{equation} \label{lickk}
K[Q] (L) = i^{wr(L)}  (\frac{z}{w})^{n-1} \sum_{k=1}^n \gamma^{k-1}\widehat{E_k} \sum_\pi i^{-wr(\pi L)}Q(\pi L). 
\end{equation}
All terms are analogous to the terms involved in (\ref{lickd}).

\smallbreak
In Theorems~\ref{doft} and~\ref{kofq} the basic invariants $T(w,a)$ and $Q(w,a)$ could be replaced by specializations of the Dubrovnik and the Kauffman polynomial respectively and, then, the invariants $D[T]$ and $K[Q]$ can be regarded as generalizations of these specialized polynomials.  For example, if $a=1$ then $Q(w,1)$ is the Brandt--Lickorish--Millett--Ho polynomial \cite{BLM} and if $w= A+A^{-1}$ and $a= -A^3$ then $Q$ becomes the Kauffman bracket polynomial \cite{kau2}. In both cases the invariant $K[Q]$ generalizes these polynomials. Furthermore, a formula analogous to (\ref{ktod}) relates the  generalized invariants $D[T]$ and $K[Q]$, see (\ref{psitoxi}). In these theorems we have formulated the invariants $D[T]$ and $K[Q]$ as 4-variable invariants. As with $H[R]$,  $D[T]$ and $K[Q]$ are topologically equivalent to the 3-variable invariants $D[D]$ and $K[K]$, via the combinatorial formulae of the invariants given in \S~\ref{generalregkd}. The 4-variable formulation is useful for keeping track of the two types of skein operations.

\begin{rem} \rm 
Since the Lickorish-type combinatorial formulae for $H[R]$, $D[T]$ and $K[Q]$  are themselves link invariants and we prove by induction that they satisfy the corresponding two-tiered skein relations, these combinatorial formulae can be used as a mathematical basis for $H[R]$, $D[T]$ and $K[Q]$. We have chosen to work out the skein theory of the generalized invariants from first principles, in order to investigate how skein theory applies to these new invariants. However, a reader of this paper may wish to  read the proofs of the Lickorish-type formulae and understand the skein relations on that basis.
\end{rem}

In \cite{kaula2} we define associated state sum models for the new invariants. These state sums are based on the skein template algorithm for the Homflypt, Kauffman and Dubrovnik polynomials. Our state sums use the skein calculation process for the invariants, but have a new property in the present context. They have a double level due to the combination in our invariants of a skein calculation combined with the evaluation of a specific invariant on the knots that are at the bottom of the skein process. If we choose a state sum evaluation of a different kind for this specific invariant, then we obtain a double-level state sum of our new invariant. This is articulated in \cite{kaula2} and we speculate in \cite{kaula2}  about possible applications for these ideas. See \S~\ref{directions}. In \cite{kaula2} we discuss the context of statistical mechanics 
models and partition functions in relation to multiple level state summations. 

\smallbreak

 The paper is organized as follows: In \S~\ref{sectheta} we detail on the algebraic construction of the new skein invariants $\Theta_d$ and $\Theta$.  
 In \S~\ref{generalregh}  we place the regular isotopy counterparts of the invariants $\Theta_d$ and $\Theta$ in a more general skein-theoretic context and we produce a full skein theory and a 4-variable invariant $H[R]$ generalizing $R$. We proceed with constructing in \S~\ref{generalregkd}  analogous skein-theoretic generalizations $D[T]$ and $K[Q]$ for the Dubrovnik and the Kauffman polynomials. In \S~\ref{generalregh}  we also provide the ambient isotopy reformulations of all new invariants. 
In \S~\ref{closedforms} we adapt the closed combinatorial formula of Lickorish  to our more general regular isotopy setting for $H[R]$ and we prove analogous Lickorish-type closed formulae for the generalized invariants $D[T]$ and $K[Q]$. In this section we also show, by using the combinatorial formulae, that the 4-variable polynomials $H[R]$, $D[T]$ and $K[Q]$ are in fact topologically equivalent to the 3-variable polynomials $H[H]$, $D[D]$ and $K[K]$ respectively. 
Finally, after summarizing in \S~\ref{conclusions},  in  \S~\ref{directions} we discuss further mathematical directions for the research in this paper. We also discuss possible relationships with reconnection in vortices in fluids, strand switching and replication of DNA, particularly the possible relations with the replication of Kinetoplast DNA, and we discuss the possibility of multiple levels in the quantum Hall effect where one considers the braiding of quasi-particles that are themselves physical 
subsystems composed of multiple electron vortices centered about magnetic field lines.

\section{Previous work} \label{sectheta}

In \cite{jula2} 2-variable framed link invariants $\Gamma_{d,D}$ were constructed for each $d\in \N$ via the Yokonuma--Hecke algebras $\YH(u)$, the Juyumaya trace and  specializations imposed on the framing parameters of the trace, where $D$ is any non-empty subset of $\Z/d\Z$.  When restricted to classical links, seen as links with zero framings on all components, these invariants give rise to ambient isotopy  invariants for classical links $\Delta_{d,D}$.  We note that for $d=1$ the algebra ${\rm Y}_{1,n}$ coincides with the Iwahori--Hecke algebra of type $A$, the trace coincides with the Ocneanu trace and the  invariant $\Delta_{1,\{1\}}$ coincides with the Homflypt polynomial, $P$. The invariants $\Delta_{d,D}$ were studied in \cite{jula3,chla}, especially their relation to  $P$, but topological comparison had not been possible due to algebraic and diagrammatic difficulties. 

 Eventually, in \cite{chmjakala,chjukala} another presentation using a different quadratic relation for the Yokonuma--Hecke algebra was adopted from \cite{chpa1} and the classical link invariants related to the new presentation of the  Yokonuma--Hecke algebras were now denoted $\Theta_{d,D}$. For $d=1$, $\Theta_{1,\{1\}}$ also coincides with $P$ with variables related to the corresponding different presentation of the Iwahori--Hecke algebra. Consequently, in \cite{chjukala} a series of results were proved, which led to the topological identification of the invariants $\Theta_{d,D}$ and to their generalization to a new 3-variable ambient isotopy invariant $\Theta$.  Firstly, it was shown that the invariants $\Theta_{d,D}$ can be enumerated only by $d$ and so they were denoted as $\Theta_d$. It was also shown that on {\it knots} the invariants  $\Theta_d$ are topologically equivalent to the Homflypt polynomial. Namely, if $K$ is a \textit{knot}, then
\begin{center}
$\Theta_d (q,z)(K) = P(q, d z)(K)$.
\end{center}
 The above result was generalized to unions of unlinked knots. Namely, if ${\mathcal K}^r := \sqcup_{i=1}^r K_i$ is a  union of $r$ unlinked knots, we have 
\begin{center}
$\Theta_d (q,z)({\mathcal K}^r) = 1/d^{1-r} P(q,d z)({\mathcal K}^r)$.
\end{center}
It was further shown in \cite{chjukala} that the invariants $\Theta_d$ satisfy on any oriented link diagram $L$ a {\it mixed skein relation} on crossings between different components of  $L$: 
\begin{center}
$
\frac{1}{\sqrt{\lambda_d}} \, \Theta_d(L_+) - \sqrt{\lambda_d} \, \Theta_d(L_-) = (q-q^{-1}) \, \Theta_d(L_0),
$
\end{center}
where $L_+$, $L_-$, $L_0$ is an oriented Conway triple and $\lambda_d :=\frac{d z  - (q-q^{-1})}{d z}$. The above skein relation is identical to the skein relation of the Homflypt polynomial $P$ considered at variables $(q, \lambda_d)$. As a consequence, the invariants $\Theta_d$ can be computed  directly from the diagram $L$ by applying the mixed skein relation between pairs of different components and gradually decomposing $L$ into  unions of unlinked knots that result as mergings of components of $L$ via the smoothings in the mixed skein relation. Then, one has to evaluate the Homflypt polynomials of the  unions of unlinked knots. Namely:
\begin{center}
$
\Theta_d(L) =  \sum_{k=1}^{c} \frac{1}{d^{1-k}} \sum_{\ell \in \mathcal{K}^k} p(\ell) \,P(\ell),
$
\end{center} 
where $c$ is the number of components of the link $L$, $\mathcal{K}^k$ denotes the set of all split links $\ell$ with $k$ split components, obtained from $L$ by applying the mixed skein relation for $k=1,\ldots, c$ and $p(\ell)$ are the coefficients coming from the application of the mixed skein relation. Finally, the above results enabled in \cite{chjukala} the topological distinction of the invariants $\Theta_d$ from the Homflypt polynomial  on Homflypt-equivalent pairs of {\it links}.
 To summarize, the family of  invariants $\{\Theta_d(q,\lambda_d)\}_{d\in\N}$ is a family of new skein invariants for links that includes  the Homflypt polynomial $P$ for $d=1$ and  are distinct from $P$ for each $d > 1$. In \cite{chjukala} it is further demonstrated that the family of invariants $\{\Theta_d(q,\lambda_d)\}_{d\in\N}$ generalizes to {\it a new 3-variable skein link invariant} $\Theta(q,\lambda,E)$, which is defined skein-theoretically on link diagrams by the following inductive rules:
\begin{enumerate}
\item On crossings involving different components the following skein relation holds:
$$
\frac{1}{\sqrt{\lambda}} \Theta( \includegraphics[scale=0.5]{crossing_pos.pdf} ) - \sqrt{\lambda} \Theta(\includegraphics[scale=0.5]{crossing_neg.pdf} ) = (q-q^{-1})\, \Theta(\includegraphics[scale=0.5]{crossing_zero.pdf}),
$$
\item For  $\mathcal{K}^r := \sqcup_{i=1}^r K_i$, a  union of  $r$ unlinked knots, with $r \geq 1$, it holds that:
$$
\Theta(\mathcal{K}^r) =  E^{1-r} \,P(\mathcal{K}^r).
$$
\end{enumerate}
The invariant $\Theta$   specializes to $P$ for $E = 1$ and to $\Theta_d$ for  $E = 1/d$,   and is stronger than $P$. Further, $\Theta$ satisfies the same properties as the invariants $\Theta_d$ and $P$, namely: multiplicative behaviour on connected sums,  inversion of certain variables on mirror images, non-distinction of mutants. For details see \cite{chmjakala}. The well-definedness of $\Theta$ is proved in \cite{chjukala} by  comparing it to an invariant $\overline{\Theta}$ for tied links, constructed from the algebra of braids and ties \cite{AJ1}, but using now the new quadratic relation for it. The invariant $\overline{\Theta}$ is analogous but, as computational evidence indicates, not the same as the invariant $\overline{\Delta}$  for tied links of F.~Aicardi and J.~Juyumaya \cite{AJ2}. In the next section of this paper we will derive an independent purely skein-theoretic proof for the well-definedness of $\Theta$. Moreover, in \cite[Appendix B]{chjukala}  W.B.R. Lickorish proved the following closed combinatorial formula for the invariant $\Theta$ on an oriented link $L$  with $n$ components (proved also  in \cite{pawa} with different methods): 
\begin{equation} \label{lickorish}
\Theta (L) = \sum_{k=1}^n \mu^{k-1}E_k \sum_\pi \lambda^{\nu(\pi)}P(\pi L),
\end{equation}
where the second summation is over all partitions $\pi$ of the components of $L$ into $k$ (unordered) subsets and $P(\pi L)$ denotes the product of the Homflypt polynomials of the $k$ sublinks of $L$ defined by $\pi$. Furthermore, $\nu(\pi)$ is the sum of all linking numbers of pairs of components of $L$ that are in distinct sets of $\pi$, $E_k = (E^{-1} - 1)(E^{-1} - 2) \cdots (E^{-1} - k + 1)$, with $E_1 =1$, and $\mu = \frac{\lambda^{-{1/2}} - \lambda^{{1/ 2}}}{q - q^{-1}}$. We see from this combinatorial formula that the extra strength of $\Theta_d$ and $\Theta$ comes from its ability to detect   linking numbers and  non-triviality  of certain sublinks of the link $L$. In our regular isotopy formulation, even the linking numbers are not needed.

\smallbreak

While the Homflypt polynomial is linked with the knot algebras discussed in this section, it is an open and interesting question to have a similar full connection to  an appropriate knot algebra for the invariants $D[T]$ and $K[Q]$, such as the framization of the BMW algebra proposed in \cite{jula4,jula5}. See also \cite{AJ3}.

\section{Skein theory for the generalized invariants}

In this section we establish the generalized invariants  $H[R]$, $D[T]$ and $K[Q]$ by using purely skein theory, without reference to the closed formulas of \S~\ref{closedforms}. This puts our invariants on the same logical basis as the original skein polynomials that they generalize.

\subsection{Skein theory for the polynomial $H[R]$ - Proof of Theorem~\ref{hofr}} \label{generalregh}

In the Introduction we have given the skein relation for the  regular isotopy invariant for oriented links,  $H[R]$, which generalizes the regular isotopy version of the Homflypt polynomial, $R$, and we have explained how to use the formalism of this skein relation to calculate this invariant. It is the purpose of this section to give a rigorous proof by induction for the well-definedness of  $H[R]$, thereby proving Theorem~\ref{hofr}. The 4-variable invariant $H[R]$ in its ambient isotopy version generalizes the skein-theoretic concept of the 3-variable invariant $\Theta(q,\lambda,E)$ \cite{chjukala}. We then use inductive methods based on the methods of Lickorish-Millett and Kauffman and our own variations to prove that  $H[R]$ is well-defined. Once this theorem is in place and by normalizing $H[R]$ to obtain its ambient isotopy counterpart, $P[R]$, we have a skein-theoretic proof  of the well-definedness of the invariant $\Theta$.

\subsubsection{Computing algorithm for $H[R]$} \label{algorithm}

Assuming Theorem~\ref{hofr}, the invariant $H[R]$ on any oriented link diagram $L$ can be easily computed by applying the following algorithm:

\begin{enumerate}[{\bf Step 1.}]
\item {\it (Diagrammatic level)} Order the components of $L$ and choose a basepoint and a direction on each component (could be the one pointed by the orientation of the component). Start  from the chosen point of the first component and  go along it in the chosen direction. When arriving at a {\it mixed} crossing for the first time along an under-arc we switch it by  the mixed skein relation,  so that  we pass by the mixed crossing along the over-arc. At the same time we smooth the mixed crossing, obtaining a new diagram in which the two components of the crossing merge into one. We repeat for all mixed crossings of the first component. In the end among all resulting diagrams there is only one with the same number of crossings as the initial diagram and in this one this component gets unlinked from  the rest and lies above all of them. The other resulting diagrams have one less crossing and have the first component fused together with some other component.  
We proceed similarly with the second component switching all its mixed crossings except for crossings involving the first component. In the end the second component gets unliked from all the rest and lies below the first one and above all others in the maximal crossing diagram, while we also obtain diagrams containing mergings of the second component with others (except component one). We continue in the same manner with all components in order and we also apply this procedure to all product diagrams coming from smoothings of mixed crossings. In the end  we obtain the unlinked version of $L$ plus a linear sum of  links $\ell$ with unlinked components resulting from the mergings of different components.
 \smallbreak
\item {\it (Computational level)}   On the level of the invariant $H[R]$, rule (1) of Theorem~\ref{hofr} tells us how the switching of mixed crossings is controlled: 
 \smallbreak
\noindent $
 H[R](L_+) - H[R](L_-) = z \, H[R](L_0) \hfill \textit{Rule (1)}
$
 \smallbreak
\noindent where $L_+$, $L_-$, $L_0$ is an oriented Conway triple. After all applications of the mixed skein relation we have a linear sum of  links $\ell$ with unlinked components. The evaluation of the invariant $H[R]$ on each $\ell$  reduces to the evaluation $R(\ell)$ by rule (2) of Theorem~\ref{hofr}:
 \smallbreak
\noindent $
  H[R]({\ell}) =  E^{1-r} \, R(\ell) \hfill \textit{Rule (2)}
$
 \smallbreak
\noindent where  $r$ is the number of knotted components of $\ell$.
 \smallbreak
\noindent In the end we obtain a linear sum of the values of the invariant $H[R]$ on all the resulting  links $\ell$ with unlinked components:
 \smallbreak
\noindent $
  H[R](L) =  \sum_{k=1}^{c} E^{1-k} \sum_{\ell \in \mathcal{K}^k} p(\ell) \,R(\ell), \hfill \textit{Rule (3)}
$
 \smallbreak
\noindent where $p(\ell)$ are the coefficients coming from the applications of the mixed skein relation. Then, on each $R(\ell)$ rule (R5) applies and then rules (R1)--(R5) are employed.
\end{enumerate}

\subsubsection{Our terminology and notations}\label{notations} 

As usual, an {\it oriented link} is a link with an orientation specified for each component. Also, a {\it link diagram} is a projection of a link on the plane with only double points, the crossings, of which there are finitely many, and which are endowed with information `under/over'. We shall be using the same notation for a link and a diagram of it as long as there is no risk of confusion.
 Two oriented link diagrams are considered {\it equivalent} if they differ by oriented regular isotopy moves on the plane, namely, by planar isotopy and by Reidemeister moves~II and~III with all variations of orientations.  For a  union  of $r$ unlinked knots, with $r \geq 1$, we will be using the notation ${\mathcal K}^r := \sqcup_{i=1}^r K_i$.

 An (oriented) link diagram is called {\it generic} if it is {\it ordered}, that is, an order $c_1, \ldots, c_r$ is given to its components, {\it directed}, that is, a direction is specified on each component, and {\it based}, that is, a basepoint is specified on each component, distinct from the double points of the crossings.  In a  generic diagram we label the mixed crossings by distinct natural numbers. The set of all generic link diagrams is denoted  ${\mathcal D}$. 

A generic diagram on $r$ components is said to be a {\it  descending stack} if, when walking along the components of ${\mathcal K}^r$ in their given order following the orientations and starting from their basepoints, every mixed crossing is first traversed along its over-arc.  Clearly, the structure of a descending stack no longer depends on the choice of basepoints; it is entirely determined by the order of its components. Note also that  a descending stack  is equivalent to the corresponding split link  ${\mathcal K}^r$  comprising the $r$ knotted components of the initial diagram. The  descending stack of knots associated to a given generic link diagram $L$ by applying the computing  algorithm on $L$ is denoted $d L$.  Furthermore, the {\it distance} of  $L$ from  $d L$ is the number of mixed crossing switches needed  to arrive at  $d L$ by applying the computing algorithm. Clearly, the distance of  a generic diagram is well-defined. 

We let, now, $\mathcal{Z} := \Z[z, w, a^{\pm 1}, E^{\pm 1}]$ denote the ring of finite Laurent polynomials in four variables $z, w, a, E$. In all proofs, the evaluation $H[R](L) \in \mathcal{Z}$ on a generic link diagram $L$ will be shortened to $(L)$. Moreover, let  $\varepsilon$ denote the sign of a mixed crossing in $L$. Then rule~(1) of Theorem~\ref{hofr} can be re-written as: 
\begin{equation}\label{mixedskein}
(L_{\varepsilon}) = (L_{-\varepsilon})  + \varepsilon z \, (L_0).
\end{equation}
Note that, if the generic diagram $L_{\varepsilon}$ has distance $n$ from $d L_{\varepsilon}$ and (\ref{mixedskein}) is applied on $L_{\varepsilon}$ within the computing algorithm, then the  diagrams $L_{-\varepsilon}$ and $L_0$ have distances less than $n$ from their associated descending stacks, if we assume that they are generic and their generic characteristics are inherited from $L_{\varepsilon}$. Note finally that the distance of a descending stack is zero.

\subsubsection{Proof of Theorem~\ref{hofr}} \label{proofregh}

For proving Theorem~\ref{hofr} one has to show that the computation of $H[R] (L)$, via the algorithmic steps described above, does not depend on the sequence of mixed crossing switches, the ordering of components, the choice of basepoints, and the performance of Reidemeister moves II and III. We use strong induction on the distance in the set  $\mathcal{D}$ of generic diagrams, as supposed to the total number of crossings.  On knots we assume the well-definedness of $R$. 
The set of all generic diagrams with distance at most $n$ is denoted $\mathcal{D}_n$. 
Clearly, the basis of the induction is the set of all descending stacks of knots, and rule~(2) of Theorem~\ref{hofr} applies.  Namely, for  a  descending stack of knots, $ \mathcal{K}^r$, define $H[R](\mathcal{K}^r) = E^{1-r} \, R(\mathcal{K}^r)$.   We then assume that the statement is valid for all generic link diagrams in $\mathcal{D}_{n-1}$, independently of choices made during the evaluation process and of Reidemeister III moves and Reidemeister II moves that do not increase the distance above $n-1$. Our aim is to prove that the statement is valid for all diagrams in the set $\mathcal{D}_n$, independently of choices and Reidemeister II and III moves not increasing the distance above $n$.

 Our proof follows the logic of \cite{LM} but adapted to our setting. The fact that the process treats self-crossings and mixed crossings differently is the main difference from \cite{LM} and it causes the need of  special  arguments in proving invariance of the resulting evaluation under the sequence of mixed crossing switches and the order of components (Propositions~\ref{orderxings} and~\ref{ordercpts}). For example, when we have to switch two mixed crossings which belong to the same pair of components, the  switching of the second crossing cannot apply also to the smoothing of the first one, since that one caused the merging of the two components. So in this diagram the second crossing is a self-crossing, and by the algorithm we cannot operate on it. So, when we change the order of the mixed crossing switches, the diagrams involved are not comparable in an obvious way as is the case in \cite{LM}. This problem is taken care of in the proof of Proposition~\ref{orderxings}, by reaching down to the corresponding descending stacks.

\vspace{.2cm}

\noindent {\it The inductive hypothesis} $(n-1)$: Assume that for any link diagram $N \in \mathcal{D}_{n-1}$ a unique polynomial $H[R](N) \in \mathcal{Z}$ is associated, which is independent of the choices made during its evaluation, is invariant under  Reidemeister III moves and  Reidemeister II moves, not increasing the distance  beyond $n-1$, and which satisfies formula~(1) of Theorem~\ref{hofr}. As a consequence, note that if $L_{\varepsilon}$ has distance  $n$ from $d L_{\varepsilon}$ then the right-hand side of (\ref{mixedskein}) is well-defined by the inductive hypothesis $(n-1)$. 

\vspace{.2cm}

\noindent {\it The recursive definition} $(n)$: Let $L\in \mathcal{D}_n$ a generic link diagram on $r$ components. Apply the steps of the computing algorithm employing formulae (1) and (2) of Theorem~\ref{hofr}. The inductive hypothesis $(n-1)$ applies on all generic diagrams in the skein tree, right after the first crossing switch, so the  resulting polynomial  $(L)$ in $\mathcal{Z}$ is well-defined.  However,  $(L)$ may depend on the sequence of mixed crossing switches, the order of the components  on the choice of basepoints, on Reidemeister III moves or non-increasing distance Reidemeister II moves applied on $L$.  So, we will prove a series of propositions to ensure that these possibilities will not occur, hence we will have proved the inductive hypothesis $(n)$. Let  $L^\prime$ denote throughout  the generic diagram resulting after the change each time, for which we assume the same generic characteristics as $L$.

\begin{prop} \label{orderxings}
 If the mixed crossings of $L$ that differ from those of $d L$ are switched in any sequence to achieve $d L$, then the corresponding polynomial $H[R](L)$ does not change. 
\end{prop}
\begin{proof}
 For $n=1$ there is nothing to show. Suppose now that $n\geq 2$ and that $L^\prime$ is the same diagram as $L$ with only difference the order of the mixed crossing switches. By the fact that any permutation is a product of adjacent elementary transpositions, and by induction on the length of the permutation induced by the two different orderings, it suffices to show invariance of $(L)$ under the exchange of two adjacent mixed crossing switches, say of the crossings  $i$ and $j$. We shall denote by $L^\prime$ the same generic diagram as $L$ but with the reverse order in switching $i$ and $j$. 
 Denote also $c_i$, $c_i^\prime$ the components of $L$ forming the $i$th crossing and $c_j$, $c_j^\prime$ the components forming the $j$th crossing. 
Let further $\sigma_i L$ and $s_i L$ denote the same diagrams as $L$, except that in  $\sigma_i L$ the $i$th crossing is switched and in $s_i L$  the $i$th crossing is smoothed. Same for the crossing $j$. 
Let now $\varepsilon_i , \varepsilon_j$ be the signs of the $i$th and $j$th crossings in $L$ respectively. We take first the order $i$ before $j$ and we compute using~(\ref{mixedskein}):
\begin{equation}\label{ij}
(L)  =  (\sigma_i L) + \varepsilon_i z (s_i L) 
  =   (\sigma_j \sigma_i L) + \varepsilon_j z (s_j \sigma_i L) + \varepsilon_i z (s_i L).
\end{equation}
Computing with the reverse order we obtain the expression:
\begin{equation}\label{ji}
(L^\prime)  =  (\sigma_j L^\prime) + \varepsilon_j z (s_j L^\prime) 
 =  (\sigma_i \sigma_j L^\prime) + \varepsilon_i z (s_i \sigma_j L^\prime) + \varepsilon_j z (s_j L^\prime). 
\end{equation}
All generic diagrams in the right-hand sides of (\ref{ij}) and (\ref{ji}) have distance at most $n-1$, so their polynomials are well-defined by the inductive hypothesis. Moreover, the generic diagrams $\sigma_j \sigma_i L$ and $\sigma_i \sigma_j L^\prime$ coincide, so $(\sigma_j \sigma_i L) = (\sigma_i \sigma_j L^\prime)$. {\it This is the first comparison level.} Hence, subtracting (\ref{ji}) from (\ref{ij}) we obtain: 
\begin{equation}\label{ijminusji}
(L) - (L^\prime) =  \varepsilon_j z \big[ (s_j \sigma_i L) - (s_j L^\prime) \big] + \varepsilon_i z \big[ (s_i L) - (s_i \sigma_j L^\prime)  \big].
\end{equation}
Further, the diagrams $s_j \sigma_i L$ and $s_j L^\prime$ differ only by the switching of the crossing $i$, which is either a mixed crossing in both diagrams 
 or a self-crossing in both diagrams (due to the smoothing of the crossing $j$). 
 Analogous considerations hold for the diagrams $s_i L$ and $s_i \sigma_j L^\prime$ which differ  by the switching of the crossing $j$. 

\bigbreak
Suppose first that the mixed crossings $i$ and $j$ belong originally to {\it   different pairs of components}, so  $c_i \neq c_j, c_j^\prime$. This case goes in analogy with the classical situation in \cite{LM}. In this case the crossing $i$ is a mixed crossing in the diagram $s_j \sigma_i L$ and the crossing $j$ is a mixed crossing in the diagram $s_i \sigma_j L^\prime$. So, we may use  formula (\ref{mixedskein}) {\it for switching them back}:
$$
(s_j \sigma_i L) = (s_j L) - \varepsilon_i z (s_j s_i L)  
\quad \text{and} \quad 
(s_i \sigma_j L^\prime)  =  (s_i L^\prime) - \varepsilon_j z (s_i s_j L^\prime).
$$
Substituting the above expressions in (\ref{ijminusji})  we obtain:
\begin{equation}\label{preijji}
(L) - (L^\prime) =  \varepsilon_j z \big[  (s_j L) - \varepsilon_i z (s_j s_i L)  - (s_j L^\prime) \big] + \varepsilon_i z \big[ (s_i L) - (s_i L^\prime) + \varepsilon_j z (s_i s_j L^\prime) \big].
\end{equation}
We now apply the {\it second round of cancellations } within the grouped terms, using the inductive hypothesis $(n-1)$, to finally obtain the comparison equality:
$$
(L) - (L^\prime) =  - \varepsilon_j z \varepsilon_i z (s_j s_i L) + \varepsilon_i z  \varepsilon_j z (s_i s_j L^\prime).
$$
 Since the coefficients match up we obtain equivalently:
\begin{equation}\label{ijji}
(L) - (L^\prime)  = \varepsilon_i  \varepsilon_j z^2 \big[ (s_i s_j L^\prime) -  (s_j s_i L)  \big]
\end{equation}
 and by the inductive hypothesis $(n-1)$, the right-hand side of the equation is zero and the two expressions $(L)$ and $(L^\prime)$ are equal. This is the {\it third comparison level}.

\bigbreak
 Suppose now that the mixed crossings $i$ and $j$ belong to {\it  the same pair of components} $c_i$ and $c_j$ of $L$ (recall Figure~\ref{ijjifig}). Then the crossing $i$ is a {\it self-crossing}  in the diagram $s_j \sigma_i L$ and the crossing $j$ is a {\it self-crossing} in the diagram $s_i \sigma_j L^\prime$.  So we are not allowed to apply (\ref{mixedskein}). 
 In this case we have to reach down to the descending stacks in order to compare the polynomials $(L)$ and $(L^\prime)$, since then the skein relation of $R$ can be applied for switching back the crossings  $i$ and $j$ in the diagrams $s_j \sigma_i L$ and $s_i \sigma_j L^\prime$. This process depends on the number of mixed crossings to be switched that are shared by the components $c_i$ and $c_j$. So, we proceed as follows.  

\smallbreak
We first prove the proposition for the case where in the initial generic diagrams $L$ and $L^\prime$  the switching of only the two mixed crossings $i$ and $j$ is needed. This means that $L$ and $L^\prime$ have distance 2. In this case the diagrams $\sigma_j \sigma_i L$ and $\sigma_i \sigma_j L^\prime$ both coincide with $d L$ and all diagrams in (\ref{ijminusji}) are also descending stacks of one component less than $L$, namely $r-1$, since $c_i$ and $c_j$ are merged together. So, applying on all of them the inductive hypothesis and rule~(2) of Theorem~\ref{hofr}, (\ref{ijminusji}) becomes the formula below by substituting the stack evaluations. 
\begin{equation}\label{ijminusji2}
(L) - (L^\prime) =  \varepsilon_j z E^{2-r} \big[ R(s_j \sigma_i L) - R(s_j L^\prime) \big] + \varepsilon_i z  E^{2-r}  \big[ R(s_i L) - R(s_i \sigma_j L^\prime)  \big].
\end{equation}
We are now in the position to {\it switch back} the self-crossing $i$ of the diagram $s_j \sigma_i L$ and the self-crossing $j$ of the diagram $s_i \sigma_j L^\prime$, by  applying now the skein relation on the level of  $R$: 
\begin{equation}\label{rsjLprime}
 R(s_j \sigma_i L) = R(s_j L) - \varepsilon_i w R(s_j s_i L)   \quad \mbox{and} \quad R(s_i \sigma_j L^\prime) = R(s_i L^\prime) - \varepsilon_j w R(s_i s_j L^\prime).
\end{equation}
Substituting in (\ref{ijminusji2}) we obtain:
\begin{equation}\label{preijminusji2r}
(L) - (L^\prime) =  \varepsilon_j z E^{2-r} \big[ R(s_j L) - \varepsilon_i w R(s_j s_i L) - R(s_j L^\prime) \big] + \varepsilon_i z  E^{2-r}  \big[ R(s_i L) - R(s_i L^\prime) + \varepsilon_j w R(s_i s_j L^\prime)  \big].
\end{equation}
Doing now the {\it second round of cancellations} among the grouped terms, since the diagrams $s_j L$ and $s_j L^\prime$ and also the diagrams $(s_i L$ and $s_i L^\prime$ are identical,  we obtain the equation:
$$
(L) - (L^\prime) =  \varepsilon_j z E^{2-r} \big[ - \varepsilon_i w R(s_j s_i L) \big] + \varepsilon_i z  E^{2-r}  \big[ \varepsilon_j w R(s_i s_j L^\prime)  \big].
$$
It should be now observed that in the above equalily the coefficients match even though $w$ has now also entered the picture. So we arrive at the equation:
\begin{equation}\label{ijminusji2r}
(L) - (L^\prime) =  \varepsilon_i  \varepsilon_j z w E^{2-r} \big[ R(s_i s_j L^\prime)  - R(s_j s_i L) \big],
\end{equation}
 which is of the same type as (\ref{ijji}). We finally proceed with the {\it third comparison level} by applying $R$ on the identical diagrams $s_i s_j L^\prime$ and $s_j s_i L$, to obtain  $(L) = (L^\prime)$.

\smallbreak
 Suppose next that in the diagrams $L$ and $L^\prime$  the switching of three mixed crossings $i$, $j$ and $k$ is needed and that the order of $i$ and $j$ is the only difference between  $L$ and $L^\prime$. The switching of $k$ applies initially on the diagrams $\sigma_j \sigma_i L$ and $\sigma_i \sigma_j L^\prime$ in  (\ref{ij}) and (\ref{ji}), independently of which pair of components it belongs to. So, we obtain: 
 \begin{equation}\label{ijk}
 (L) = (\sigma_k \sigma_j \sigma_i L) + \varepsilon_k z (s_k \sigma_j \sigma_i L) + \varepsilon_j z (s_j \sigma_i L) + \varepsilon_i z (s_i L),
 \end{equation}
 \begin{equation}\label{jik}
 (L^\prime) = (\sigma_k \sigma_i \sigma_j L^\prime) + \varepsilon_k z  (s_k \sigma_i \sigma_j L^\prime)  + \varepsilon_i z (s_i \sigma_j L^\prime) + \varepsilon_j z (s_j L^\prime). 
 \end{equation}
 The diagrams involved in (\ref{ijk}) and (\ref{jik}) are all descending stacks. Further, by the inductive hypothesis $(n-1)$, $(\sigma_k \sigma_j \sigma_i L) = (\sigma_k \sigma_i \sigma_j L^\prime)$ and  $(s_k \sigma_j \sigma_i L) = (s_k \sigma_i \sigma_j L^\prime)$,  and they have in (\ref{ijk}) and (\ref{jik}) the same coefficients. Subtracting and applying the {\it first cancellation round} we obtain the following equation, in which   terms with the same coefficients are grouped together: 
 \begin{equation}\label{ijkminusjik}
 (L) - (L^\prime) =   \varepsilon_j z \big[ (s_j \sigma_i L) - (s_j L^\prime) \big]
  + \varepsilon_i z \big[ (s_i L) - (s_i \sigma_j L^\prime) \big], 
 \end{equation}
 which is of the same type as (\ref{ijminusji}). The grouped pairs differ by one crossing switch and in all four diagrams involved in (\ref{ijkminusjik}) the components $c_i$ and $c_j$ are merged together. So, the crossing $k$ is either a self-crossing in all of them, which means that it is a crossing between the same pair of components as $i$ and $j$,  or a mixed crossing in all of them, which means that at least one of the components forming $k$ is different from $c_i$ and $c_j$. 

 {\it If $k$ is a self-crossing} in the four diagrams in (\ref{ijkminusjik}), then these diagrams are already descending stacks of $r-1$ components.  We then apply rule~(2) of Theorem~\ref{hofr} and the skein relation of $R$ for the $i$ and $j$ crossing respectively and we proceed as with (\ref{ijminusji2})--(\ref{ijminusji2r}) above, to obtain the equation:
\begin{equation}\label{ijminusjiselfk}
(L) - (L^\prime) =  \varepsilon_i  \varepsilon_j z w E^{2-r} \big[ R(s_i s_j L^\prime)  - R(s_j s_i L) \big],
\end{equation}
 which is analogous to (\ref{ijminusji2r}) and (\ref{ijji}). We finally proceed with the {\it third round of cancellation} by applying $R$ on the identical diagrams $s_i s_j L^\prime$ and $s_j s_i L$, to obtain  $(L) = (L^\prime)$.

  {\it If $k$ is a mixed crossing } in the four diagrams in  (\ref{ijkminusjik})  we proceed with switching it in all four diagrams so as to obtain descending stacks. Namely: 
$$
	 \begin{array}{lcl}
 (L) - (L^\prime) & = &  
   \varepsilon_j z \big[ (\sigma_k s_j \sigma_i L) + \varepsilon_k z (s_k s_j \sigma_i L) - (\sigma_k s_j L^\prime) - \varepsilon_k z (s_k s_j L^\prime) \big] \\ 
 & + & \varepsilon_i z \big[ (\sigma_k s_i L) + \varepsilon_k z (s_k s_i L) 
 - (\sigma_k s_i \sigma_j L^\prime)  - \varepsilon_k z (s_k s_i \sigma_j L^\prime) \big]. 
 	 \end{array}
$$
Grouping pairs of polynomials with the same coefficients we obtain equivalently: 
 \begin{equation}\label{ijkjikk}
 \begin{array}{lclcl}
 (L) - (L^\prime) & = &  
 \varepsilon_j z \big[ (\sigma_k s_j \sigma_i L) - (\sigma_k s_j L^\prime) \big]
 & + & 
 \varepsilon_j \varepsilon_k z^2 \big[ (s_k s_j \sigma_i L) - (s_k s_j L^\prime) \big] \\ 
  & + & 
 \varepsilon_i z 
  \big[ (\sigma_k  s_i L) - (\sigma_k  s_i \sigma_j L^\prime) \big] 
 & + &
  \varepsilon_i \varepsilon_k z^2 \big[ (s_k  s_i L) - (s_k  s_i \sigma_j L^\prime) \big].
 \end{array}
 \end{equation}
 Now, all diagrams in the above expression are descending stacks either of $r-1$ components, if one smoothing is involved, or of $r-2$ components, if two smoothings are involved. So, applying rule~(2) of Theorem~\ref{hofr} we obtain:
 \begin{equation}\label{ijkjikkstacks}
 \begin{array}{lclcl}
 (L) - (L^\prime) & = &  
 \varepsilon_j z E^{2-r} \big[ R (\sigma_k s_j \sigma_i L) -  R (\sigma_k s_j L^\prime) \big]
 & + & 
 \varepsilon_j \varepsilon_k z^2  E^{3-r} \big[ R  (s_k s_j \sigma_i L) -  R (s_k s_j L^\prime) \big] \\ 
  & + & 
 \varepsilon_i z E^{2-r}  
  \big[  R (\sigma_k  s_i L) -  R (\sigma_k  s_i \sigma_j L^\prime) \big] 
 & + &
  \varepsilon_i \varepsilon_k z^2  E^{3-r} \big[  R (s_k  s_i L) -  R (s_k  s_i \sigma_j L^\prime) \big].
 \end{array}
 \end{equation}
In (\ref{ijkjikkstacks}) each grouped pair differs by one self-crossing switch. We are now in the position to {\it switch back} the self-crossings $i$ and $j$ in the terms of the paired diagrams that contain $\sigma_i$ or $\sigma_j$, but using the skein relation on the level of the invariant  $R$. This is  {\it the second comparison level}.  Namely, (\ref{ijkjikkstacks}) yields equivalently:
 \begin{equation}\label{ijkjikkr}
 \begin{array}{lclcl}
 (L) - (L^\prime) & = &  
 \varepsilon_j z \, E^{2-r} \big[ R(\sigma_k s_j L) - \varepsilon_i w R(\sigma_k s_j s_i L) - R(\sigma_k s_j L^\prime) \big] \\
 & + & 
 \varepsilon_j \varepsilon_k z^2 \, E^{3-r} \big[ R(s_k s_j L) - \varepsilon_i w R(s_k s_j s_i L) - R(s_k s_j L^\prime) \big] \\ 
  & + & 
 \varepsilon_i z \, E^{2-r} 
  \big[ R(\sigma_k  s_i L) - R(\sigma_k  s_i L^\prime) + \varepsilon_j w R(\sigma_k  s_i s_j L^\prime)  \big] \\ 
 & + &
  \varepsilon_i \varepsilon_k z^2 \, E^{3-r} \big[ R(s_k  s_i L) - R(s_k  s_i L^\prime) + \varepsilon_j w R(s_k  s_i s_j L^\prime) \big].
 \end{array}
 \end{equation}
After the obvious cancellations of terms  (\ref{ijkjikkr}) becomes equivalently:
$$
 \begin{array}{lclcl}
 (L) - (L^\prime) & = &  
 \varepsilon_j z \, E^{2-r} \big[ - \varepsilon_i w R(\sigma_k s_j s_i L) \big] 
  +  
 \varepsilon_j \varepsilon_k z^2 \, E^{3-r} \big[ - \varepsilon_i w R(s_k s_j s_i L) \big] \\ 
  & + & 
 \varepsilon_i z \, E^{2-r} 
  \big[  \varepsilon_j w R(\sigma_k  s_i s_j L^\prime)  \big] 
 + 
  \varepsilon_i \varepsilon_k z^2 \, E^{3-r} \big[ \varepsilon_j w R(s_k  s_i s_j L^\prime) \big],
 \end{array}
$$
which by grouping further terms with matching coefficients becomes:
$$
 \begin{array}{lclcl}
 (L) - (L^\prime) & = &  
 \varepsilon_i \varepsilon_j z  w \, E^{2-r} \big[ R(\sigma_k  s_i s_j L^\prime) -  R(\sigma_k s_j s_i L) \big] 
	\\ 
  & + & 
   \varepsilon_i \varepsilon_j \varepsilon_k z^2 w \, E^{3-r} \big[  R(s_k  s_i s_j L^\prime) -  R(s_k s_j s_i L) \big].
 \end{array}
$$
We finally {\it switch back} crossing $k$ in the diagrams $\sigma_k  s_i s_j L^\prime$ and $\sigma_k s_j s_i L$ using the skein relation of $R$ and we obtain the comparison equation:
 \begin{equation}\label{thirdkmixed}
 (L) - (L^\prime)  =   
 \varepsilon_i \varepsilon_j z  w \, E^{2-r} \big[ R( s_i s_j L^\prime) -  R( s_j s_i L) \big] 
	 + 
   \varepsilon_i \varepsilon_j \varepsilon_k z^2 w \, E^{3-r} \big[  R(s_k  s_i s_j L^\prime) -  R(s_k s_j s_i L) \big],
 \end{equation}
which is analogous to (\ref{ijminusjiselfk}), (\ref{ijminusji2r}) and (\ref{ijji}).
 We have now reached {\it the third comparison level} and we finally obtain $(L) = (L^\prime)$. 

\smallbreak
  The proof  for any number of mixed crossing switches proceeds in the same manner as for three. Namely, if we have more than three mixed crossing switches, we would start from  Equations~\ref{ijk} and~\ref{jik} and we would apply a switch at the diagrams with highest number of crossings. The point is that we will always have {\it first level cancellations} by the inductive hypothesis $(n-1)$ and the remaining comparison equation will always be of the same type as (\ref{ijminusji}) and (\ref{ijkminusjik}). Starting from there, we have to distinguish cases for every new mixed crossing in the remaining diagrams in the difference equation, as we did for $k$. The principle is that we have to calculate far enough so as to arrive at descending stacks where we apply rule 2. Recall (\ref{ijminusji2}) and (\ref{ijkjikkstacks}). 	Now all expressions become $R$-evaluations and then the difference equations become $R$-identities. More precisely, we apply the skein rule of the $R$-polynomial  in order to {\it switch back} the self-crossings $i$ and $j$ in all terms appearing in the difference equation. This now leads to the {\it second comparison  level}, after which  we switch back, if needed, the crossings subsequent  to $i$ and $j$, using the skein rule of $R$. We are then at the {\it third comparison  level}, with equations of type (\ref{ijminusjiselfk}) and (\ref{thirdkmixed}) accordingly,	which yield  $(L) = (L^\prime)$. It is impressive that the coefficients in the pairings match, and this is ensured by the process. The procedure is abstracted in the diagram below: 
	
$$
\begin{diagram} 
\node{}\node{} \node{} \node{}  \node{(i,j) \ vs \ (j,i)} \arrow{sw,t}{same \ pair \ cpts} \arrow{se,t}{different \  pairs \  cpts} \node{}  
\\ 
\node{} \node{} \node{} \node{}  \arrow{sw,t}{distance > 2} \arrow{se,t}{distance \ 2} \node{}  \node{comparison}  
\\ 
\node{} \node{}  \node{}  \arrow{sw,t}{distance > 3} \arrow{se,t}{distance \ 3}  \node{}  \node{desc. \ stacks}   \node{}   
\\ 
\node{} \node{}  \arrow{sw,..}\arrow{se,..}  \node{}  \node{desc. \ stacks}   \node{} \node{}  
\end{diagram}
$$

\vspace{1.5cm}

\noindent Thus, the proof of the proposition is concluded. 
\end{proof}

Independence of the evaluation of the polynomial $(L)$ under changes of directions (recall \S~\ref{notations}) and choices of basepoints is now an immediate consequence of Proposition~\ref{orderxings}. 

\begin{cor} \label{direction} 
The change of direction  on any component is irrelevant in the computation of the polynomial $H[R](L)$.
\end{cor}

\begin{cor} \label{basepoints}
 The polynomial $H[R](L)$ does not depend on the choice of basepoints. 
\end{cor}

The proof for both corollaries follows by noticing that, in changing direction of a component or in moving a basepoint, the only possible change is the sequence of mixed crossing switches. Indeed, suppose a basepoint is moved  before or after a crossing, say $k$ (See Figure~\ref{basepoint}). Then we have the following possibilities: if $k$ is a self-crossing then it remains a self-crossing, while if $k$ is a mixed crossing to be switched (or not) before, then it will also be switched (or not) after. So, the only difference is the order of switching mixed crossings, invariance under which is ensured by Proposition~\ref{orderxings}. 

\smallbreak
\begin{figure}[!ht]
\begin{center}
\includegraphics[width=10cm]{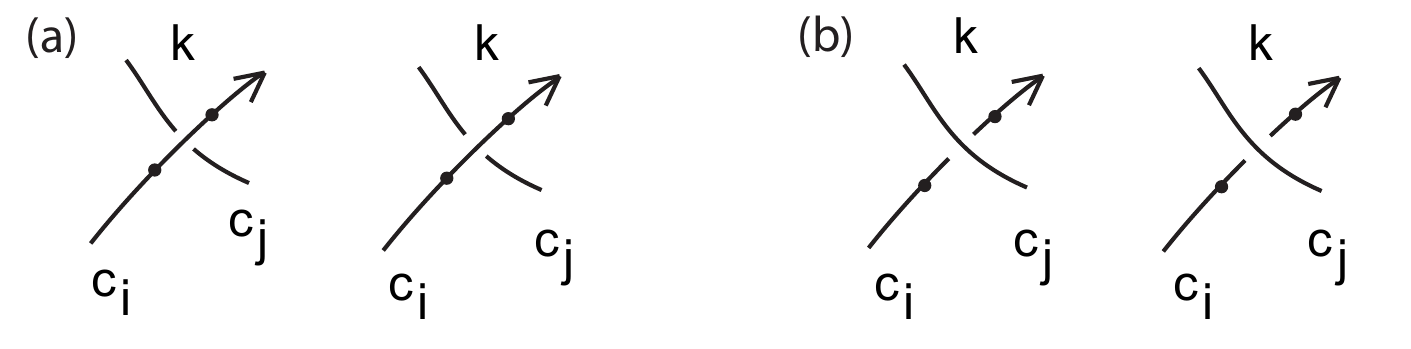}
\caption{(a) moving basepoints on an over-arc (b) moving basepoints on an under-arc}
\label{basepoint}
\end{center}
\end{figure}

\begin{prop} \label{skeinrule}
 The polynomial $H[R] (L)$ satisfies the skein relation (1) of Theorem~\ref{hofr} on mixed crossings not increasing the distance. 
\end{prop}
\begin{proof}
We observe that the computer algorithm reduces distance in all diagrams involved in the skein tree for arriving at $dL$ from $L$ starting from the very first application of rule (1) on $(L)$, so they have well-defined polynomials by the inductive hypothesis $(n-1)$. More generally, if any mixed crossing of $L$ is switched, such that the distance does not increase, then  the diagrams involved in the skein relation have well-defined polynomials. Recall also the discussion after Eq.~\ref{mixedskein}. 
\end{proof}                                                                                             

\begin{prop} \label{reidem}
 The polynomial $H[R](L)$ is invariant under Reidemeister III moves  and Reidemeister II moves  that do not increase the distance beyond $n$. 
\end{prop}
\begin{proof}
{\it Type II.} \ View Figure~\ref{reidemII}. The case $i=j$ means that the move takes place on the same component $c_i$, so the move is visible only on the level of $R$, which is known to be a link invariant.

 In the case $j<i$ in the left-hand illustration of Figure~\ref{reidemII} no mixed crossing gets switched by the algorithm. So, applying this move on all diagrams created in the skein tree, we have by the inductive hypothesis $(n-1)$ that $H[R]$ remains invariant on these diagrams and their equivalent counterparts. Hence, $(L^\prime) = (L)$. 

\smallbreak
\begin{figure}[H] 
\begin{center}
\includegraphics[width=7.5cm]{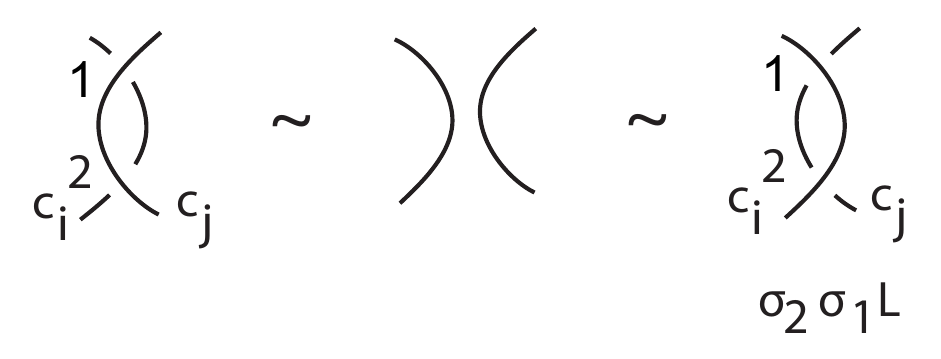}
\caption{The Reidemeister II moves}
\label{reidemII}
\end{center}
\end{figure}

Let now $i<j$ in the left-hand illustration of Figure~\ref{reidemII}. This means that, by the algorithm, the two mixed crossings  in the figure will be switched. Note that the two crossings have opposite signs independently of the choices of orientations.  By Proposition~\ref{orderxings}, the two mixed crossings can be switched first and their order is irrelevant, so we  label them 1 and 2. Let now $\varepsilon$ be the sign of crossing 1. Then we distinguish two cases:

If $\varepsilon = +1$ we have:
$
(L)  = (\sigma_1 L) + z (s_1 L)  
  =  (\sigma_2 \sigma_1 L) - z (s_2 \sigma_1 L) + z (s_1 L). $
Examining Figure~\ref{reidemIIproof}(a) we see that the diagrams  $s_2 \sigma_1 L$ and $s_1 L$ differ just by planar isotopy. So we have $(s_2 \sigma_1 L) = (s_1 L)$.  Hence, $(L) = (\sigma_2 \sigma_1 L)$ and the situation is reduced to the case $j<i$.

\smallbreak
\begin{figure}[H]
\begin{center}
\includegraphics[width=9cm]{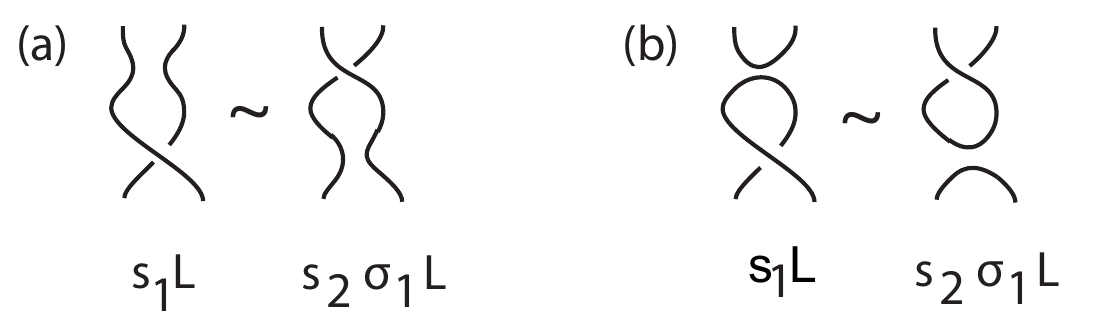}
\caption{Invariance of $H[R]$ under the Reidemeister II moves}
\label{reidemIIproof}
\end{center}
\end{figure}

If $\varepsilon = -1$ we have from Figure~\ref{reidemIIproof}(b):
$
(L)  =  (\sigma_1 L) - z (s_1 L)  
  =   (\sigma_2 \sigma_1 L) + z (s_2 \sigma_1 L) - z (s_1 L). 
$
The diagrams $s_2 \sigma_1 L$ and $s_1 L$ both contain a positive kink.  Now, this kink is on the same component, since the two components $c_i$ and $c_j$ are merged into one, and of the same handedness. So, the two diagrams  are regular isotopic and the isotopy is only visible on the level of $R$, which is known to be a link invariant. Hence, by the inductive hypothesis $(n-1)$  we conclude $(s_2 \sigma_1 L) = (s_1 L)$. So, $(L) = (\sigma_2 \sigma_1 L)$ and, again, the situation is reduced to the case $j<i$.
 Hence, the end polynomials computed before and after the move are the same also in this case.

\smallbreak

\noindent {\it Type III.} \  Suppose that no mixed crossing switch is  involved, so the algorithm does not see the move. This can happen, for example, in the case  $k = j = i$ and the invariance of $H[R]$ under the move rests on the invariance of $R$. Also in the case $ k\leq j \leq i$ but with not all arcs in the same component. Then there is nothing to do. View Figure~\ref{reidemIII1}, but  ignore for the time  the crossing marked with 1 and a shaded disc.

Suppose now that one mixed crossing switch is required. This means that not all arcs belong to the same component. Assume that $j < k$ and $j < i$ and that the mixed crossing marked with a shaded disc in Figure~\ref{reidemIII1} has to be switched.  By  Proposition~\ref{orderxings} we can label this crossing by 1 before and after the move is performed. The two diagrams with crossing 1 switched that differ by one Reidemeister III move  fall now in the previous case. We shall follow  what happens to the two diagrams with the corresponding smoothings of the marked crossing. Note that the crossing in question retains its sign after the move is performed. Hence the computations are the same  throughout both skein trees deriving from the diagrams with the mixed crossing switched.  In  Figure~\ref{reidemIIIpf1} we see the two possibilities according to the compatibility of orientations. In both cases the components $c_j$ and $c_k$ merge into one. In  case (a) the resulting diagrams are planar isotopic. In case (b) they differ by two Reidemeister II moves not increasing the distance. Hence, by the inductive hypothesis $(n-1)$, both configurations will be assigned the same polynomials.

\smallbreak
\begin{figure}[H]
\begin{center}
\includegraphics[width=7cm]{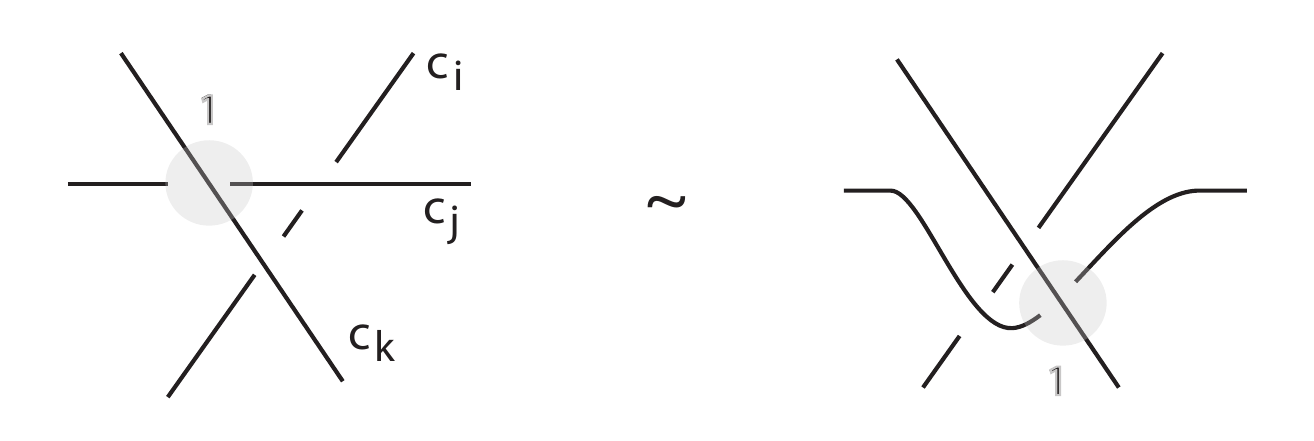}
\caption{The Reidemeister III moves: case 1}
\label{reidemIII1}
\end{center}
\end{figure}

\smallbreak
\begin{figure}[H]
\begin{center}
\includegraphics[width=12cm]{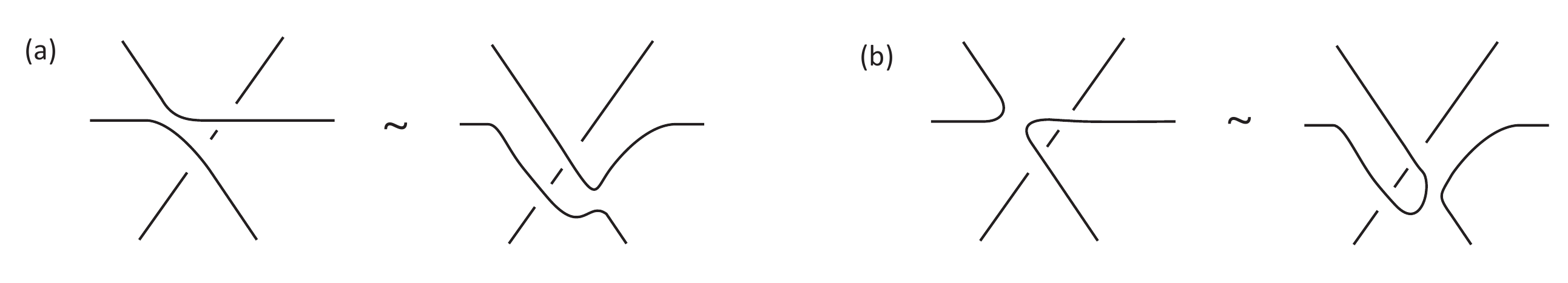}
\caption{Invariance of $H[R]$ under the Reidemeister III moves - case 1}
\label{reidemIIIpf1}
\end{center}
\end{figure}

Suppose now that two mixed crossing switches will be needed. This means that not all arcs belong to the same component. Assume that $i <  k \leq j$. Then the crossings marked with  shaded  discs in Figure~\ref{reidemIII2} have to be switched. Note that the two crossings retain their signs after the move is performed. By Proposition~\ref{orderxings} we can label these crossings by 1 and 2 before and after the move. 

\smallbreak
\begin{figure}[H]
\begin{center}
\includegraphics[width=6.7cm]{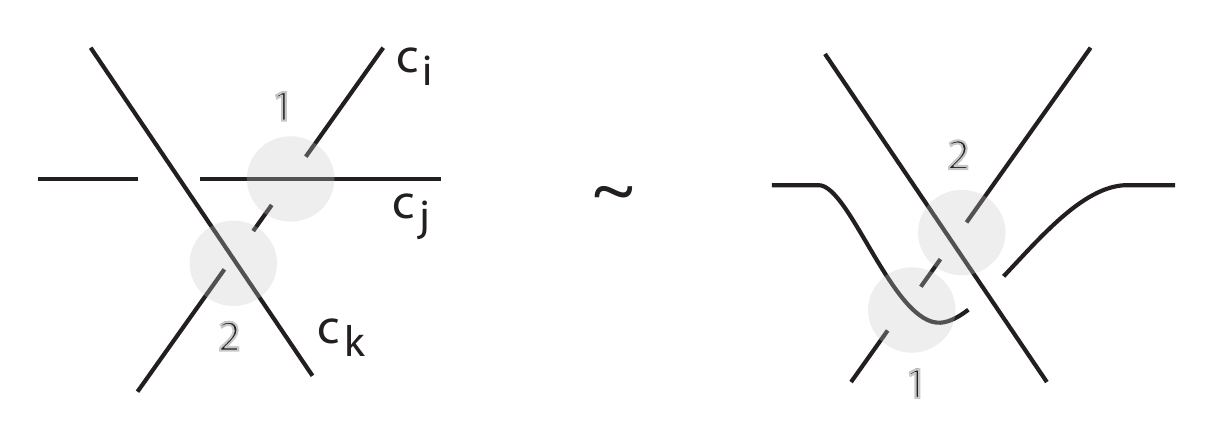}
\caption{The Reidemeister III moves: case 2}
\label{reidemIII2}
\end{center}
\end{figure}

We first apply skein rule (1) on crossing 1 on both sides of the move.  The analysis of the two resulting diagrams with the crossing switched reduces to the case where only one mixed crossing has to be switched (crossing 2). Let us follow now what happens to the two diagrams with the smoothings of crossing 1. 
 Figure~\ref{reidemIIIpf2} illustrates the two possibilities according to the compatibility of orientations. In  case (a) the resulting diagrams are planar isotopic. In case (b) they differ by two non increasing distance Reidemeister II moves. In both cases the components $c_i$ and $c_j$ merge into one, so the isotopy is only `seen' on the level of $R$. So, by the inductive hypothesis $(n-1)$, both configurations will be assigned the same polynomials. Hence, the proof is concluded. 

\smallbreak
\begin{figure}[H]
\begin{center}
\includegraphics[width=12cm]{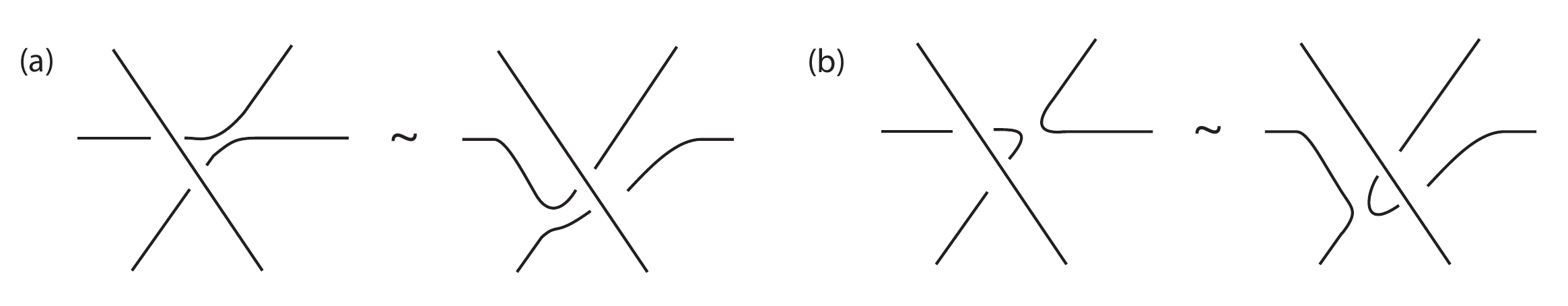}
\caption{Invariance of $H[R]$ under the Reidemeister III moves - case 2 }
\label{reidemIIIpf2}
\end{center}
\end{figure}

It is worth noting here that the logic we followed for the proof of this last case with two crossing switches would not have worked if the order of the crossings in question were reversed. Indeed, if crossing 2 were to be switched first, no Reidemeister III move would be available on the initial diagram. \end{proof}



\begin{prop} \label{ordercpts}
 The polynomial $H[R](L)$ is independent of the order of the components of $L$.
\end{prop}
\begin{proof}
Suppose  that a different order of components is assigned to $L$ and let $L^\prime$ denote the resulting generic diagram. Note that, in general, $L^\prime$ will have different distance from $d L^\prime$ than $L$ from $d L$. By a similar argument as in the proof of Proposition~\ref{orderxings}, it suffices to prove the proposition in the case where the relative positions of only one pair of adjacent components, say $A$ and $B$,  is switched. By Proposition~\ref{orderxings} we may assume that the mixed crossings between $A$ and $B$ are last in the two sequences.  So, we have done the unlinking of all pairs of components of $L$ and $L^\prime$ except for the pair $A, B$. Moreover, in both resolution trees the coefficients in the computations toward $(L)$ and $(L^\prime)$ are identical up to the point where we have to start unlinking  $A$ and $B$.   For the resulting diagrams in both skein trees  we apply the inductive hypothesis, whereby the order of the components $A$ and $B$ does not affect the values of their polynomials. 

\smallbreak

Let us  denote the two maximal crossing diagrams in the above two resolution trees by $AB$ and $BA$ respectively. In $AB$ component $A$ is prior to $B$ in the order of $L$, while in $BA$ component $B$ is prior to $A$ in the order of $L^\prime$. By the above reasoning we may assume that $L = AB$ and $L^\prime = BA$. 

Let $r+r^\prime$ be the total number of mixed crossings between components $A$ and $B$ and let $r$ be the number of mixed crossings between $A$ and $B$ that need to be switched in $AB$ so that the two components get unlinked and $A$ lies above $B$. By Proposition~\ref{orderxings} we may assign to these crossings the numbers $1, \ldots, r$ and let  $\varepsilon_1, \ldots, \varepsilon_r$ be their signs. After switching them all we end up with the final descending stack, for which we will use the notation  $\frac{A}{B}$ and in which $A$ is before $B$ in the given order and it lies above $B$. 
 Further, $r^\prime$ is the number of mixed crossings between $B$ and $A$ that need to be switched in $BA$ so that the two components get unlinked and $B$ lies above $A$. We may assign to these crossings the numbers $r+1, \ldots, r+r^\prime$ and we let  $\varepsilon_{r+1}, \ldots, \varepsilon_{r+r^\prime}$ be their signs.  After switching them all  we end up with the descending stack for which we will use the notation $\frac{B}{A}$ and in which  $B$ lies above $A$ and it comes before $A$. 

More precisely,  we start by selecting in $AB$ the crossing numbered 1 and we rename $AB$ to $L_{\varepsilon_1}$. After switching crossing~1 we select the crossing numbered 2 and we proceed similarly until we reach crossing~$r$. After switching crossing $r$ we obtain the diagram  $L_{-\varepsilon_r}$, which is in fact $\frac{A}{B}$. Namely, we have the sequence of generic diagrams with component $A$ coming before $B$:
\begin{equation} \label{1tor}
\begin{array}{lclcl}
(AB) & := & (L_{\varepsilon_1})  & = & (L_{-\varepsilon_1})  + \varepsilon_1\, z \, (L_{0,1}), \\
 (L_{-\varepsilon_1})  & := & (L_{\varepsilon_2}) & = & (L_{-\varepsilon_2})  + \varepsilon_2\, z \, (L_{0,2}), \\
 & \vdots & & & \\
(L_{-\varepsilon_{r-1}}) & := & (L_{\varepsilon_{r}}) & = & (L_{-\varepsilon_{r}})  + \varepsilon_{r}\, z \, (L_{0,r}) \\ 
&  &  & = & (\frac{A}{B})  + \varepsilon_{r}\, z \, (L_{0,r}).
\end{array}
\end{equation}
At the same time we select in  $BA$ the crossing $r+1$ for switching, so we rename $BA$ to $L^\prime_{\varepsilon_{r+1}}$. Then, we rename $L^\prime_{-\varepsilon_{r+1}}$ to $L^\prime_{\varepsilon_{r+2}}$ and we select in it the crossing $r+2$. Proceeding in this manner we arrive at the final step of the process that yields the descending stack   $\frac{B}{A}$, indicated below as $L^\prime_{-\varepsilon_{r+ r^\prime}}$, after switching the crossing numbered $r+ r^\prime$ in the diagram $L^\prime_{-\varepsilon_{r+ r^\prime - 1}}$ (renamed to $L^\prime_{\varepsilon_{r+ r^\prime}}$). Namely, we have the sequence of generic diagrams with component $B$ coming before $A$:
\begin{equation} \label{rplus1torprime}
\begin{array}{lclcl}
(BA) & := & (L^\prime_{\varepsilon_{r+1}})  & = & (L^\prime_{-\varepsilon_{r+1}})  + \varepsilon_{r+1} \, z \, (L^\prime_{0, r+1}), \\
 (L^\prime_{-\varepsilon_{r+1}}) & := & (L^\prime_{\varepsilon_{r+2}})  & = & (L^\prime_{-\varepsilon_{r+2}})  + \varepsilon_{r+2} \, z \, (L^\prime_{0, r+2}), \\
 & \vdots & & & \\
(L^\prime_{-\varepsilon_{r+ r^\prime - 1}}) & := & (L^\prime_{\varepsilon_{r+ r^\prime}})  & = & (L^\prime_{-\varepsilon_{r+ r^\prime}})  + \varepsilon_{r+ r^\prime}\, z \, (L^\prime_{0, r+ r^\prime}) \\
&  &  & = & (\frac{B}{A})  + \varepsilon_{r+ r^\prime}\, z \, (L^\prime_{0, r+ r^\prime}).
\end{array}
\end{equation}
Substituting now the expressions in (\ref{1tor}) consecutively, starting from the last equation, we obtain:
  \begin{equation} \label{ABtoAoverB}
(AB) = \left( \frac{A}{B} \right) + \, z  \, \left[ \varepsilon_1 \, (L_{0,1})  +  \cdots  + \varepsilon_{r}\, (L_{0,r}) \right].
\end{equation}
Analogously, from (\ref{rplus1torprime})  we obtain:
  \begin{equation} \label{BAtoBoverA}
(BA) = \left( \frac{B}{A} \right) + \, z  \, \left[ \varepsilon_{r+1} \, (L^\prime_{0,r+1}) + \cdots + \varepsilon_{r+ r^\prime}\, (L^\prime_{0,r+ r^\prime}) \right].
\end{equation}
 Denoting now: 
\begin{equation} \label{XandY}
(X) := \varepsilon_1 \, (L_{0,1})  +  \cdots + \varepsilon_{r}\, (L_{0,r}) \quad {\rm and } \quad (Y) :=  \varepsilon_{r+1} \, (L^\prime_{0,r+1}) + \cdots + \varepsilon_{r+ r^\prime}\, (L^\prime_{0,r+ r^\prime}),
\end{equation}	 
Eqs.~\ref{ABtoAoverB} and~\ref{BAtoBoverA} are shortened to the following:
  \begin{equation} \label{ABX}
(AB) = \left( \frac{A}{B} \right) + \, z  \, (X),
\end{equation}
  \begin{equation} \label{BAY}
(BA) = \left( \frac{B}{A}  \right) + \, z  \, (Y).
\end{equation}

Subtracting equations (\ref{ABX}) and (\ref{BAY}) by parts we obtain:
\begin{equation} \label{ABvsBAXY}
 (AB) - (BA) = \left(\frac{A}{B} \right) - \left(\frac{B}{A} \right) + z \, [(X) - (Y)].
\end{equation}
Further, we observe that  the descending stacks  $\frac{A}{B}$ and $\frac{B}{A}$ are both assigned the same value of $H[R]$ since, by the basis of induction  we have: 
\begin{equation} \label{reductiontor}
\left(\frac{A}{B} \right) = E^{1-c} \, R\left(\frac{A}{B} \right) \quad \& \quad \left(\frac{B}{A} \right) = E^{1-c} \, R\left(\frac{B}{A} \right),
\end{equation}
where $c$ is the number of components in both descending stacks. But, by the well-definedness of the link invariant $R$ it is ensured that 
$R(\frac{A}{B}) = R(\frac{B}{A})$. So $ (\frac{B}{A}) = (\frac{A}{B})$. So, (\ref{ABvsBAXY}) becomes:
\begin{equation} \label{ABBAonlyXY}
 (AB) - (BA) =  z \, [(X) - (Y)].
\end{equation}

In order to prove further that $(X) = (Y)$ we argue as follows: we do the same procedure as above, switching and smoothing progressively all  $r$ crossings starting from $AB$ and all $r^\prime$ crossings starting from $BA$, but this time applying the skein relation of the invariant $R$.  We obtain equations of the same form as (\ref{1tor}) and (\ref{rplus1torprime}), but now $z$ is replaced by $w$ and the invariant $R$ is evaluated on all diagrams. Summing up we obtain: 
\begin{equation} \label{evalR}
 R(AB) - R(BA) =  w \, [R(X) - R(Y)],
\end{equation}
where $R(X) := \varepsilon_1 \, R(L_{01})  +  \cdots + \varepsilon_{r}\, R(L_{0r})$  and  $R(Y) :=  \varepsilon_{r+1} \, R(L^\prime_{0,r+1}) + \cdots + \varepsilon_{r+ r^\prime}\, R(L^\prime_{0,r+ r^\prime})$. Clearly, by the well-definedness of $R$ we have $R(AB) = R(BA)$.
 Note, now, that all intermediate generic diagrams in (\ref{1tor}) and (\ref{rplus1torprime}) that come from smoothings are descending stacks of $c-1$ components,  since the components $A$ and $B$ have merged into one. So, by the induction basis:
\begin{equation} \label{RonLi}
(L_{0,i}) = E^{2-c} \, R(L_{0,i}),
\end{equation}
for all $i=1,\ldots, r+ r^\prime$. Multiplying then Eq.~\ref{evalR} by $E^{2-c}$ we obtain:
\begin{equation} \label{equalXY}
(X) - (Y) = 0.
\end{equation}
Substituting (\ref{equalXY}) in (\ref{ABBAonlyXY})   we  finally obtain $ (AB) =  (BA)$ and the proof of the Proposition is concluded.
\end{proof}

By Propositions \ref{orderxings}, \ref{skeinrule}, \ref{reidem}, \ref{ordercpts} and Corollaries~\ref{direction} and \ref{basepoints}  the proof of Theorem~\ref{hofr} is now completed. \hfill {\it Q.E.D.}

\begin{rem} \rm
The places in the proof of Theorem~\ref{hofr} where the actual properties of the invariant $R$ were intrinsically used, beyond rule (2) of the Theorem, were in the proofs of Propositions~\ref{orderxings} and~\ref{ordercpts}, where it was essential that $R$ satisfies the same form of skein relation as $H[R]$. However, nowhere in the proof was it forced that $R$ has the same indeterminate $z$ as  $H[R]$. 
\end{rem}

\subsubsection{Translation to Ambient Isotopy} \label{generalambp}

In this subection we define and discuss the ambient isotopy generalized invariant, counterpart of the regular isotopy generalized invariant $H[R]$ constructed above.
 Let $P$ denote the classical Homflypt polynomial. Then, as we know, one can obtain the ambient isotopy invariant $P$ from its regular isotopy counterpart $H$ via the formula: 
$$
P (L) := a^{-wr (L)} H (L),
$$ 
where $wr (L)$ is the total writhe of the oriented diagram $L$. 
Analogously, and letting $G$ denote $P$ but with different variable, from our generalized regular isotopy invariant $H[R]$ one can derive an ambient isotopy invariant $P[G]$ via:
\begin{equation}\label{prfromhr}
P[G] (L) :=  a^{-wr (L)} H[R] (L).
\end{equation}

Then for the invariant $P[G]$ we have the following:

\begin{thm} \label{pofr}
Let $P (z,a)$ denote the Homflypt polynomial and let $G (w,a)$ denote  the same invariant but with a different parameter $w$ in place of $z$.  Then there exists a unique ambient isotopy invariant of classical oriented links $P[G]: \mathcal{L} \rightarrow \Z[z, w, a^{\pm 1}, E^{\pm 1}]$ defined by the following rules:
\begin{enumerate}
\item On crossings involving different components the following skein relation holds:
$$
 a  \, P[G](L_+) - {a}^{-1} \, P[G](L_-) = z \, P[G](L_0),
$$
where $L_+$, $L_-$, $L_0$ is an oriented Conway triple.
\item For ${\mathcal K}^r := \sqcup_{i=1}^r K_i$, a union of $r$ unlinked knots, with $r \geq 1$, it holds that:
$$
P[G]({\mathcal K}^r) =  E^{1-r} \, G({\mathcal K}^r).
$$
\end{enumerate}
\end{thm}

\begin{rem} \label{specializations} \rm 
As pointed out in the Introduction, in Theorem~\ref{hofr}  we could specialize the $z$, the $w$, the $a$ and the $E$ in any way we wish. For example, if $a=1$ then $R (w,1)$ becomes the Alexander--Conway polynomial, while if $w = \sqrt{a} - 1/\sqrt{a}$ then $R (\sqrt{a} - 1/\sqrt{a} , a)$ becomes the unnormalized Jones polynomial. In each case $H[R]$ can be regarded as a generalization of that polynomial. 
Furthermore, in the case where $G(w,a) = P(z,a)$  (for $w=z$) the ambient isotopy invariant $P[P]$ coincides with the new 3-variable link invariant $\Theta(q, \lambda, E)$ \cite{chjukala}, while  for $w=z$ and  $E=1/d$, $P[P]$ coincides with  the invariant $\Theta_d$ \cite{jula2} (for $E=1$ it coincides with $P$), recall \S~\ref{sectheta}. So, our invariant $P[G]$ is stronger than $P$ and it is a (seemingly) 4-variable generalization of the invariant  $\Theta$. As we shall see below (Proposition~\ref{topequivh}) one variable is redundant. Hence, our proof of the existence of  $H[R]$ provides a direct skein-theoretic proof of the existence of the invariant $\Theta$, without the need of algebraic tools or the theory of tied links. 
 Finally, for $w=z = \sqrt{a} - 1/\sqrt{a}$ the invariant $P[P]$ can be renamed to $V[V]$, $V$ denoting the  ambient isotopy version of the Jones polynomial, and it coincides with the new 2-variable link invariant $\theta(a, E)$ \cite{goula2}, which generalizes and is stronger than $V$. 
\end{rem}

\subsection{Generalization of the Dubrovnik and the Kauffman polynomials} \label{generalregkd}

In the Introduction we have given the skein relations for the regular isotopy invariants for unoriented  links, $D[T]$ and $K[Q]$, which generalize the Dubrovnik polynomial, $D$, and the Kauffman polynomial, $K$, respectively, and we have explained how to use the formalism of these skein relations to calculate these invariants. It is the purpose of this section to give a rigorous proof for the well-definedness of   $D[T]$ and $K[Q]$, thereby proving Theorems~\ref{doft} and \ref{kofq}.

We recall the notations $\mathcal{L}^u$ for the class of unoriented links and $\mathcal{Z} := \Z[z, w, a^{\pm 1}, E^{\pm 1}]$ for the ring of finite Laurent polynomials in four indeterminates $z, w, a, E$. 

For proving Theorems~\ref{doft} and \ref{kofq} we keep the same notations and we follow the same method as in \S~\ref{generalregh}. So, we will avoid repetitions and we will only elaborate on the differences in the proofs. The computing algorithm for $D[T]$ and $K[Q]$ is analogous to the one in \S~\ref{algorithm} for $H[R]$, where in Step~2 the rules  of Theorem~\ref{doft} and of Theorem~\ref{kofq} now apply for $D[T]$ and $K[Q]$  respectively.

\subsubsection{Skein theory for the polynomial $D[T]$ - Proof of Theorem~\ref{doft}}\label{generalregd}

 In the proofs that follow, the evaluation $D[T](L) \in \mathcal{Z}$ on a generic link diagram $L \in \mathcal{L}^u$ will be shortened to $(L)$. 

Moreover, let  $\varepsilon$ denote the type of a mixed crossing in $L$. Then rule~(1) of Theorem~\ref{doft} can be re-written as: 
\begin{equation}\label{dmixedskein}
(L_{\varepsilon}) = (L_{-\varepsilon})  + \varepsilon \, z \, \big[(L_0) - (L_{\infty})\big].
\end{equation}

Adapting Proposition~\ref{orderxings} and Corollary~\ref{direction} to $D[T]$, for a given mixed crossing $i$ of a diagram $L \in \mathcal{L}^u$ we denote by $\sigma_i L$ the same diagram but with that crossing switched and by 
$$
s_i L := a_i L - b_i L,
$$ 
the formal difference of the  diagrams $a_i L$ and $b_i L$, which are the same as $L$  but with the $A$- and $B$-smoothing respectively replacing crossing $i$. In this notation we have the polynomial equation  $(s_i L) = (a_i L) - (b_i L)$. Then the proof carries through with the same formal expressions as in Proposition~\ref{orderxings}, concluding with equality of $(L)$ and $(L^\prime)$.

\smallbreak
 Corollary~\ref{basepoints} and Proposition~\ref{skeinrule} adapt directly for  $D[T]$. Concerning now invariance under the Reidemeister moves, we adapt the proof of Proposition~\ref{reidem} and we first check Reidemeister~II moves. Looking at the left-hand instance of Figure~\ref{reidemII} for the case $i<j$, we note first that the two mixed crossings to be switched are of opposite type. Proceeding the analysis we obtain the final equation:
\begin{equation}\label{dreidemII}
(L) = (\sigma_2 \sigma_1 L) - z \, \big[ (a_2 \sigma_1 L) - (b_2 \sigma_1 L) \big] + z \, \big[ (a_1 L ) - (b_1 L ) \big],
\end{equation}
which involves the diagrams in Figure~\ref{kreidIIproof}. These diagrams comprise the ones involved in  the cases (a) and (b) illustrated in Figure~\ref{reidemIIproof} and  the right-hand diagram of Figure~\ref{reidemII}. Then, by the same arguments as in Proposition~\ref{reidem} we obtain invariance of $D[T]$ under the Reidemeister~II moves.

\smallbreak
\begin{figure}[H]
\begin{center}
\includegraphics[width=7cm]{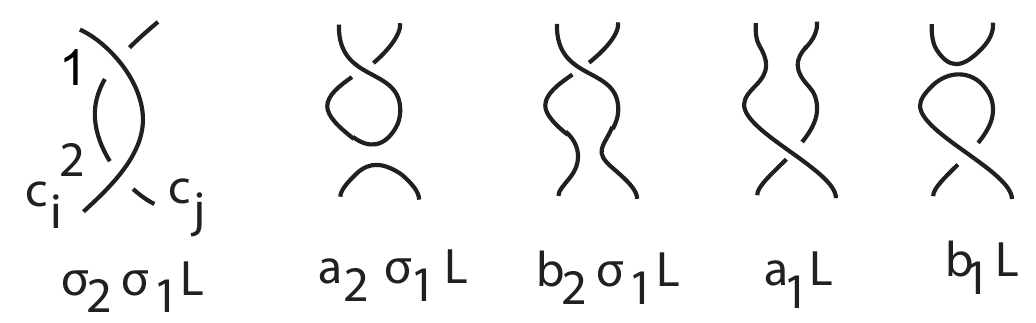}
\caption{The diagrams in the analysis of Reidemeister II moves for  $D[T]$}
\label{kreidIIproof}
\end{center}
\end{figure}

For proving  invariance of $D[T]$ under Reidemeister~III moves we also follow the same strategy as for $H[R]$ in Proposition~\ref{reidem}. In the cases where one (resp. two) of the crossings involved in the move need to be switched, all four configurations of Figure~\ref{reidemIIIpf1} (resp. Figure~\ref{reidemIIIpf2}) will enter the picture and the same arguments apply as for $H[R]$, see Figure~\ref{kreidIIIproof1}.

\smallbreak
\begin{figure}[H]
\begin{center}
\includegraphics[width=11cm]{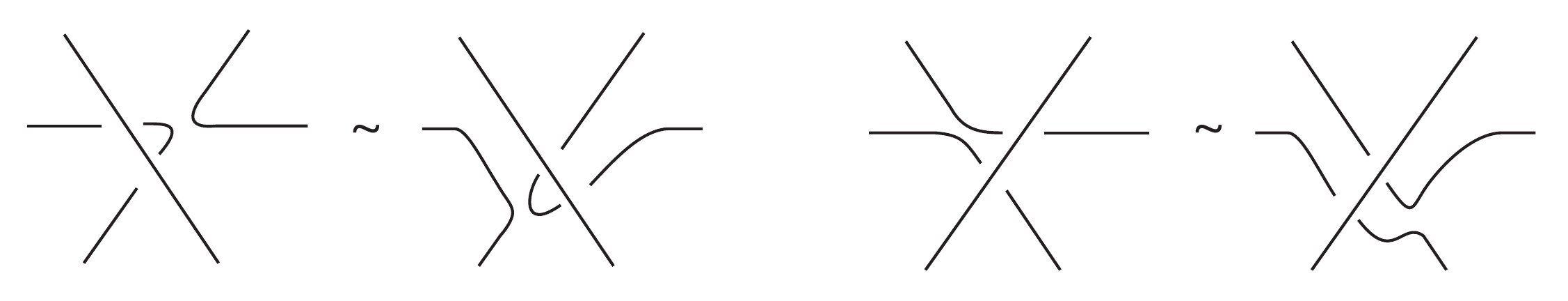}
\caption{Analysis of the Reidemeister III moves for  $D[T]$}
\label{kreidIIIproof1}
\end{center}
\end{figure}

We will finally prove independence of $D[T]$ under ordering of components. We use the same notations as in the proof of Proposition~\ref{ordercpts}, but with different interpretations due to (\ref{dmixedskein}).  Then, with the conventions above and  denoting: 
$$
L_{0,i} := a_i L_{\varepsilon_i} - b_i L_{\varepsilon_i} \qquad \mbox{and} \qquad 
L^\prime_{0,j} := a_j L^\prime_{\varepsilon_j} - b_j L^\prime_{\varepsilon_j},
$$
for $i = 1, \ldots, r$ and  $j = r+1, \ldots, r+ r^\prime$,  we have the same formal expressions as in  Proposition~\ref{ordercpts}, where further the invariant $R$ is replaced by $T$.
 With the above analysis the proof of Theorem~\ref{doft} is now concluded. \hfill \qed

\subsubsection{Translation to Ambient Isotopy} \label{generalamby}

In this subection we define  the ambient isotopy generalized invariant, counterpart of the regular isotopy generalized invariant $D[T]$ constructed above.
 Let $Y$ denote the classical ambient isotopy Dubrovnik polynomial. Then, one can obtain the ambient isotopy invariant $Y$ from its regular isotopy counterpart $D$ via  the formula:  
$$
Y (L) := a^{-wr (L)} D(L),
$$ 
where $wr (L)$ is the total writhe of the diagram $L$ for some choice of orientation of $L$. 
Analogously, and letting $Z$ denote $Y$ but with different variable, from our generalized regular isotopy invariant $D[T]$ one can derive an ambient isotopy invariant $Y[Z]$ via:
\begin{equation}\label{yzfromdt}
Y[Z] (L) :=  a^{-wr (L)} D[T] (L).
\end{equation}

In order to have a skein relation one leaves it in regular isotopy form.

\subsubsection{Skein theory for the polynomial $K[Q]$ - Proof of Theorem~\ref{kofq}} \label{generalregk}

We shall now establish the generalization of the Kauffman polynomial by means of proving  Theorem~\ref{kofq}.  
 In the proofs that follow, the evaluation $K[Q](L) \in \mathcal{Z}$ on a generic link diagram $L \in \mathcal{L}^u$ will be shortened to $(L)$. Moreover, let  $\varepsilon$ denote the type of a mixed crossing in $L$. Then rule~(1) of Theorem~\ref{kofq} can be re-written as: 
\begin{equation}\label{kmixedskein}
(L_{\varepsilon}) = - (L_{-\varepsilon})  + z \, \big[(L_0) + (L_{\infty})\big].
\end{equation}

The symmetry of (\ref{kmixedskein}) implies that we may suppress the indication $\varepsilon$ in the computations. For this reason and also due to the difference in signs from $D[T]$ we will record carefully some computations in the proofs, since they do not carry through directly from $H[R]$ and $D[T]$.

Adapting Proposition~\ref{orderxings} to $K[Q]$, for a given mixed crossing $i$ of $L$ we denote by $\sigma_i L$ the same diagram but with that crossing switched and by 
$$
s_i L := a_i L + b_i L,
$$ 
the formal sum of the  diagrams $a_i L$ and $b_i L$, which are the same as $L$  but with the $A$- and $B$-smoothings respectively replacing the crossing $i$. In this notation we have the polynomial equation $(s_i L) = (a_i L) + (b_i L)$. 
Then, relation~(\ref{ij}) in the proof of Proposition~\ref{orderxings} is replaced by the relation:
\begin{equation}\label{kij}
\begin{array}{rcl}
(L) & = & -(\sigma_i L) + z (s_i L) 
  =   (\sigma_j \sigma_i L) -  z (s_j \sigma_i L) + z  (s_i L)
	\end{array}
\end{equation}
and relation~(\ref{ji}) is replaced by the relation:
\begin{equation}\label{kji}
\begin{array}{rcl}
 (L^\prime) & := & - (\sigma_j L) + z (s_j L) 
 =  (\sigma_i \sigma_j L) - z (s_i \sigma_j L) + z (s_j L). 
\end{array}
\end{equation}
Applying now relation~(\ref{kmixedskein}) to the link diagrams $a_i L, b_i L, a_j L, b_j L$, namely: 

\smallbreak

\begin{center}
$ (a_i L) = -(\sigma_j a_i L) + z (s_j a_i L) $
\end{center}
\begin{center}
$ (b_i L) = -(\sigma_j b_i L) + z (s_j b_i L) $
\end{center}
\begin{center}
$ (a_j L) = -(\sigma_i a_j L) + z (s_i a_j L) $
\end{center}
\begin{center}
$ (b_j L) = -(\sigma_i b_j L) + z (s_i b_j L) $
\end{center}

\smallbreak

\noindent and replacing in (\ref{kij}) and (\ref{kji}) we obtain equality of $(L)$ and $(L^\prime)$,  using the same arguments as in  Proposition~\ref{orderxings}. Corollary~\ref{basepoints} and Proposition~\ref{skeinrule} carry through directly for  $K[Q]$. Concerning now invariance under the Reidemeister moves, we adapt the proof of Proposition~\ref{reidem}. We first check the Reidemeister~II moves. As for $D[T]$, the two mixed crossings (to be switched for the case $i<j$)  at the left-hand instance of Figure~\ref{reidemII} are of opposite type. Proceeding the analysis we obtain the final equation:
\begin{equation}\label{kreidemII}
(L) = (\sigma_2 \sigma_1 L) - z \, \big[ (a_2 \sigma_1 L) + (b_2 \sigma_1 L) \big] + z \, \big[ (a_1 L ) + (b_1 L ) \big],
\end{equation}
which involves the diagrams in Figure~\ref{kreidIIproof}. Then, by the same arguments as in the proof of Proposition~\ref{reidem} we show invariance of $K[Q]$ under the Reidemeister~II moves.

For proving  invariance of $K[Q]$ under Reidemeister~III moves in the cases where one (resp. two) of the crossings involved in the move needs to be switched, all four configurations of Figure~\ref{kreidIIIproof1} will enter the picture and the same arguments apply as for $H[R]$.

We will finally adapt Proposition~\ref{ordercpts} for $K[Q]$. We use the same notations as in the proof of Proposition~\ref{ordercpts}, but with different interpretations due to (\ref{kmixedskein}).  Then, with the conventions above and by denoting: 
$$
L_{0,i} := a_i L_{\varepsilon_i} + b_i L_{\varepsilon_i} \qquad \mbox{and} \qquad 
L^\prime_{0,j} := a_j L^\prime_{\varepsilon_j} + b_j L^\prime_{\varepsilon_j},
$$
for $i = 1, \ldots, r$ and  $j = r+1, \ldots, r+ r^\prime$,  we have :
\begin{equation} \label{k1tor}
\begin{array}{lclcl}
(AB) & := & (L_{\varepsilon_1})  & = & -(L_{-\varepsilon_1}) + z \, (L_{0,1}), \\
 (L_{-\varepsilon_1}) & := & (L_{\varepsilon_2}) & = & -(L_{-\varepsilon_2}) + z \, (L_{0,2}), \\
 & \vdots & & & \\
(L_{-\varepsilon_{r-1}}) & := & (L_{\varepsilon_{r}}) & = & -(L_{-\varepsilon_{r}}) + z \, (L_{0,r})  =  -(\frac{A}{B}) + z \, (L_{0,r}).
\end{array}
\end{equation}
At the same time and selecting in  $BA$ the crossing $r+1$ we have :
\begin{equation} \label{krplus1torprime}
\begin{array}{lclcl}
(BA) & := & (L^\prime_{\varepsilon_{r+1}})  & = & -(L^\prime_{-\varepsilon_{r+1}}) + z \, (L^\prime_{0, r+1}), \\
 (L^\prime_{-\varepsilon_{r+1}}) & := & (L^\prime_{\varepsilon_{r+2}})  & = & - (L^\prime_{-\varepsilon_{r+2}}) + z \, (L^\prime_{0, r+2}), \\
 & \vdots & & & \\
(L^\prime_{-\varepsilon_{r+ r^\prime - 1}}) & := & (L^\prime_{\varepsilon_{r+ r^\prime}})  & = & -(L^\prime_{-\varepsilon_{r+ r^\prime}})  + z \, (L^\prime_{0, r+ r^\prime})  =  - (\frac{B}{A})  + z \, (L^\prime_{0, r+ r^\prime}).
\end{array}
\end{equation}
Substituting now the expressions in (\ref{k1tor}) consecutively, starting from the last equation, we obtain:
  \begin{equation} \label{kABtoAoverB}
(AB) = -\left(\frac{A}{B} \right) + \, z  \, \left[(L_{0,1}) + \cdots + (L_{0,r}) \right].
\end{equation}
Analogously, from (\ref{krplus1torprime})  we obtain:
  \begin{equation} \label{kBAtoBoverA}
(BA) = -\left(\frac{B}{A} \right) + \, z  \, \left[(L^\prime_{0,r+1}) + \cdots + (L^\prime_{0,r+ r^\prime}) \right].
\end{equation}
 Denoting now: 
\begin{equation} \label{kXandY}
(X) := (L_{0,1})  +  \cdots +  (L_{0,r}) \quad {\rm and } \quad (Y) := (L^\prime_{0,r+1}) + \cdots + (L^\prime_{0,r+ r^\prime}),
\end{equation}	 
Eqs.~\ref{kABtoAoverB} and~\ref{kBAtoBoverA} are shortened to the following:
  \begin{equation} \label{kABXY}
(AB) = -\left(\frac{A}{B} \right) + \, z  \, (X) \quad {\rm and } \quad 
(BA) = -\left(\frac{B}{A} \right) + \, z  \, (Y).
\end{equation}
Then, subtracting by parts we obtain:
\begin{equation} \label{kABvsBAXY}
 (AB) - (BA) = \left(\frac{B}{A} \right)  -  \left(\frac{A}{B} \right) + z \, [(X) - (Y)].
\end{equation}
Further, we observe that  the descending stacks  $\frac{A}{B}$ and $\frac{B}{A}$ are both assigned the same value of $K[Q]$ by the basis of induction. Thus,  we have: 
\begin{equation} \label{reductiontoq}
\left(\frac{A}{B} \right) = E^{1-c} \, Q\left(\frac{A}{B} \right) \quad \& \quad \left(\frac{B}{A} \right) = E^{1-c} \, Q\left(\frac{B}{A} \right),
\end{equation}
where $c$ is the number of components in both descending stacks. But, by the well-definedness of the link invariant $Q$ it is ensured that 
$Q(\frac{A}{B}) = Q(\frac{B}{A})$. So $ (\frac{B}{A}) = (\frac{A}{B})$. So, Eq.~\ref{kABvsBAXY}  becomes:
\begin{equation} \label{kABBAonlyXY}
 (AB) - (BA) =  z \, [(X) - (Y)].
\end{equation}

In order to prove further that $(X) = (Y)$ we apply the same procedure as above, switching and smoothing progressively all  $r$ crossings starting from $AB$ and all $r^\prime$ crossings starting from $BA$, but this time working with the invariant $Q$.  We obtain equations of the same form as (\ref{k1tor}) and (\ref{krplus1torprime}), but now $z$ is replaced by $w$ and the invariant $R$ is evaluated on all diagrams. Summing up we obtain: 
\begin{equation} \label{kevalR}
 Q(AB) - Q(BA) =  w \, [Q(X) - Q(Y)],
\end{equation}
where $Q(X) := Q(L_{0,1})  +  \cdots + Q(L_{0,r})$  and  $Q(Y) :=  Q(L^\prime_{0,r+1}) + \cdots + Q(L^\prime_{0,r+ r^\prime})$. Of course, by the well-definedness of $Q$ we have $Q(AB) = Q(BA)$.
 Now,  all intermediate generic diagrams in (\ref{k1tor}) and (\ref{krplus1torprime}) that come from smoothings  are descending stacks of $c-1$ components,  since the components $A$ and $B$ have merged into one. So, by the basis of induction:
\begin{equation} \label{kRonLi}
(L_{0,i}) = E^{2-c} \, Q(L_{0,i}),
\end{equation}
for all $i=1,\ldots, r+ r^\prime$. Multiplying then Eq.~\ref{kevalR} by $E^{2-c}$ we obtain:
\begin{equation} \label{kequalXY}
(X) - (Y) = 0.
\end{equation}
Substituting, finally, (\ref{kequalXY}) in (\ref{kABBAonlyXY})   we  obtain: 
\begin{equation} \label{kBAequalAB}
(AB) =  (BA).
\end{equation}
Hence, the proof of the Proposition is concluded and with this the proof of Theorem~\ref{kofq} is also is concluded. \hfill \qed

\subsubsection{Translation to Ambient Isotopy} \label{generalambf}

As for the Dubrovnik polynomial, in this subection we define for the Kauffman polynomial the ambient isotopy generalized invariant, counterpart of the regular isotopy generalized invariant $K[Q]$ constructed above.
 Let $K$ denote the classical regular isotopy Kauffman polynomial. Then, one can obtain the ambient isotopy invariant $F$ from its regular isotopy counterpart $K$ via  the formula:  
$$
F (L) := a^{-wr (L)} K(L),
$$ 
where $wr (L)$ is the total writhe of the diagram $L$ for some choice of orientation of $L$. 
Analogously, and letting $S$ denote $F$ but with different variable, from our generalized regular isotopy invariant $K[Q]$ one can derive an ambient isotopy invariant $F[S]$ via:
\begin{equation}\label{fsfromkq}
F[S] (L) :=  a^{-wr (L)} K[Q] (L).
\end{equation}

In order to have a skein relation one leaves it in regular isotopy form.

\section{Closed combinatorial formulae for the generalized invariants } \label{closedforms}

In this section we give closed combinatorial  formulae for our regular isotopy extension $H[R]$ and  for the  invariants $D[T]$ and $K[Q]$, generalizing the Dubrovnik and the Kauffman polynomials. 
As we mentioned in the previous section, in \cite[Appendix B]{chjukala} W.B.R. Lickorish provides an analogous closed combinatorial formula for the definition of the invariant $\Theta = P[P]$, that uses the  Homflypt polynomials and linking numbers of sublinks of a given link.

\subsection{A closed  formula for $H[R]$} \label{secphi}

We recall the basic rules for $H[R]$ from Theorem~\ref{hofr} and rules (R1)--(R5) for $R$ that are listed thereafter. We then have the following.

\begin{thm}\label{theta_linking_P}
Let $L$ be an oriented link with $n$ components and let
\begin{equation}\label{hr}
\Phi[R](L) = \left( \frac{z}{w} \right)^{n-1} \sum_{k=1}^n \eta^{k-1} \widehat{E}_k \sum_\pi R(\pi L)
\end{equation}
where the second summation is over all partitions $\pi$ of the components of $L$ into $k$ (unordered) subsets and $R(\pi L)$ denotes the product of the Homflypt polynomials of the $k$ sublinks of $L$ defined by $\pi$. Furthermore, $\widehat{E}_k = (\widehat{E}^{-1} - 1)(\widehat{E}^{-1} - 2) \cdots (\widehat{E}^{-1} - k + 1)$, with $\widehat{E} = E \frac{z}{w}$, $\widehat{E}_1 =1$, and $\eta = \frac{a - a^{-1}}{w}$. If $H[R]$ satisfies  properties (1) and (2) of Theorem~\ref{hofr} then $H[R](L) = \Phi[R](L)$.  
  Independently of that, $\Phi[R]$ is a well-defined regular isotopy invariant of oriented links and $\Phi[R]$ satisfies the relations (1) and (2) of Theorem~\ref{hofr}. Therefore $\Phi[R]$ is a closed formula for $H[R]$. 
\end{thm}

\begin{proof}
We will first show that if $H[R]$ satisfies  properties (1) and (2) of Theorem~\ref{hofr} then $H[R]$ can be given by the formula (\ref{hr}). Before proving this, note the following equalities:
\begin{align*}
R(L_1 \sqcup L_2) &= \eta \, R(L_1) \, R(L_2),
\\
H[R](L_1 \sqcup L_2) &= \frac{\eta}{E} \, H[R](L_1) \, H[R](L_2).
\end{align*}
In the case where both $L_1$ and $L_2$  are knots the above formulae follow directly from rules (R5) and (2) above. If at least one of $L_1$ and $L_2$ is a true link, then the formulae follow by doing independent skein processes on  $L_1$ and $L_2$ for bringing them down to unlinked components, and then using the defining rules above.

Suppose now that a diagram of $L$ is given. The proof is by induction on $n$ and on the number, $u$, of crossing changes between distinct components required to change $L$ to $n$ unlinked knots. If $n=1$ there is nothing to prove. So assume the result true for $n-1$ components and $u-1$ crossing changes and prove it true for $n$ and $u$.

The induction starts when $u = 0$. Then $L$ is the union of $n$ unlinked components $L_1, \dots, L_n$. A classic elementary result concerning the Homflypt polynomial shows that $R(L) = \eta^{n-1}R(L_1) \cdots R(L_n)$. Furthermore, in this situation, for any $k$ and $\pi$, $R(\pi L) = \eta^{n-k}R(L_1) \cdots R(L_n)$. Note that $H[R](L) = E^{1-n} R(L) = \eta^{n-1} E^{1-n} R(L_1) \cdots R(L_n)$. So it is required to prove that
\begin{equation}\label{hunlink}
\eta^{n-1} E^{1-n} = \left( \frac{z}{w} \right)^{n-1} \eta^{n-1} \sum_{k=1}^n S(n,k)(\widehat{E}^{-1} - 1)(\widehat{E}^{-1} - 2) \cdots (\widehat{E}^{-1} - k + 1),
\end{equation}
where $S(n,k)$ is the number of partitions of a set of $n$ elements into $k$ subsets. Now it remains to prove that:
\begin{equation}\label{hstirling}
E^{1-n} \left( \frac{z}{w} \right)^{1-n} = \widehat{E}^{1-n} = \sum_{k=1}^n S(n,k)(\widehat{E}^{-1} - 1)(\widehat{E}^{-1} - 2) \cdots (\widehat{E}^{-1} - k + 1).
\end{equation}
Now, from the theory of combinatorics, $S(n,k)$ is a Stirling number of the second kind and this required formula is a well-known result about such numbers.

Now suppose that $u > 0$. Suppose that in a sequence of $u$ crossing changes that changes $L$, as above, into unlinked knots, the first change is to a crossing $c$ of sign $\epsilon$ between components $L_1$ and $L_2$. Let $L^\prime$ be $L$ with the crossing changed and $L^0$ be $L$ with the crossing annulled. Now, from the definition of $H[R]$,
$$
H[R] (L) =  H[R] (L^\prime) + \epsilon z \, H[R] (L^0).
$$

\noindent The induction hypotheses imply that the result is already proved for $L^\prime$ and $L^0$ so
\begin{equation}\label{hstar}
H[R] (L) = \left( \frac{z}{w} \right)^{n-1} \sum_{k=1}^n \eta^{k-1}\widehat{E}_k \sum_{\pi^\prime} R(\pi^\prime L^\prime) + \epsilon z \left( \frac{z}{w} \right)^{n-2} \sum_{k=1}^{n-1} \eta^{k-1}\widehat{E}_k \sum_{\pi^0} R(\pi^0 L^0),
\end{equation}
where $\pi^\prime$ runs through the partitions of the components of $L^\prime$ and $\pi^0$ through those of $L^0$.

A sublink $X^0$ of $L^0$ can be regarded as a sublink $X$ of $L$ containing $L_1$ and $L_2$ but with $L_1$ and $L_2$ fused together by annulling the crossing at $c$. Let $X^\prime$ be the sublink of $L^\prime$ obtained from $X$ by changing the crossing at $c$. Then
$$
R (X) = R (X^\prime) + \epsilon w \, R(X^0).
$$
This means that the second (big) term in (\ref{hstar}) is
\begin{equation}\label{hbig_term}
\frac{z} {w} \left( \frac{z}{w} \right)^{n-2} \sum_{k=1}^{n-1} \eta^{k-1} \widehat{E}_k \sum_{\rho} \bigl( R(\rho L) - R(\rho^\prime L^\prime  ) \bigr),
\end{equation}
where the summation is over all partitions $\rho$ of the components of $L$ for which $L_1$ and $L_2$ are in the same subset and $\rho^\prime$ is the corresponding partition of the components of $L^\prime$.

Note that, for any partition $\pi$ of the components of $L$ inducing partition $\pi^\prime$ of $L^\prime$, if $L_1$ and $L_2$ are in the same subset then 
we can have a difference between $R(\pi L)$ and $R(\pi^\prime L^\prime)$, but when $L_1$ and $L_2$ are in different subsets then
 \begin{equation}\label{hdiff_subsets}
R(\pi^\prime L^\prime) = R(\pi L). 
\end{equation}
Thus, substituting (\ref{hbig_term}) in (\ref{hstar}) we obtain:
\begin{equation}\label{halmostdone}
H[R](L) = \left( \frac{z}{w} \right)^{n-1} \sum_{k=1}^n \eta^{k-1} \widehat{E}_k \biggl( \sum_{\pi^\prime} R(\pi^\prime L^\prime) + \sum_{\rho} \bigl( R(\rho L) - R(\rho^\prime L^\prime)\bigr) \biggr),
\end{equation}
where $\pi^\prime$ runs through all partitions of $L^\prime$ and $\rho$ through partitions of $L$ for which $L_1$ and $L_2$ are in the same subset. Note that, for $k=n$ the second sum is zero. Therefore:
\begin{equation}\label{hdone}
H[R] (L) = \left( \frac{z}{w} \right)^{n-1} \sum_{k=1}^n \eta^{k-1} \widehat{E}_k \biggl( \sum_{\pi^\prime} R(\pi^\prime L^\prime) + \sum_{\rho}  R(\rho L) \biggr),
\end{equation}
where  $\pi^\prime$  runs through only partitions of $L^\prime$ for which $L_1$ and $L_2$ are in different subsets and $\rho$ through all partitions of $L$ for which $L_1$ and $L_2$ are in the same subset.  Hence, using (\ref{hdone}) and also (\ref{hdiff_subsets}), we obtain:
$$
H[R] (L) = \left( \frac{z}{w} \right)^{n-1} \sum_{k=1}^n \eta^{k-1} \widehat{E}_k \sum_\pi R(\pi L)
$$
and the induction is complete.

Now we observe that the formula for $\Phi[R](L)$ satisfies properties (1) and (2) of Theorem~\ref{hofr}. The proof of this is actually contained in the formalism of the proof that we have already made above. The reader will see that the argument from (\ref{hstar}) to the end of the proof can be read as a verification that the formula $\Phi[R](L)$ satisfies the skein relation (1)  of Theorem~\ref{hofr} and the argument at the beginning of the proof involving formulas (\ref{hunlink}) and (\ref{hstirling}) can be regarded as a verification of property (2)  of Theorem~\ref{hofr} for the invariant $\Phi[R]$. The proof is now complete.
\end{proof}

\begin{rem} \rm 
Note that the combinatorial formula~(\ref{hr}) can be regarded by itself as a definition of the invariant $H[R]$.  In the same way the original Lickorish formula (\ref{lickorish}) can be regarded as a definition for the invariant $\Theta = P[P]$. Clearly, the two formulae for $H[H]$ and $\Theta$ are intechangeable by writhe normalization, see (\ref{prfromhr}). However, from our regular isotopy approach it becomes clear that the linking numbers of sublinks, appearing in the formula~(\ref{lickorish}), do not play an intrinsic role in the theory.
\end{rem}

\begin{rem} \rm \label{fiandh}
In an early version of this paper the formula (\ref{hr}) was originally proved for the specialization $H[H]$ of $H[R]$, where $w=z$. Namely, 
\begin{equation}\label{phi}
H[H](L)(z,a,E) = \sum_{k=1}^n \phi^{k-1} E_k \sum_\pi H(\pi L),
\end{equation}
where $\phi = \frac{a - a^{-1}}{z}$ and $E_k = (E^{-1} - 1)(E^{-1} - 2) \cdots (E^{-1} - k + 1)$, with $E_1 =1$. This formula for the invariant $H[H]$ is fully detailed in \cite{kaula2}. Konstantinos Karvounis \cite{ka2} adapted our proof to the (seemingly more general, see Proposition~\ref{topequivh} below) invariant $H[R]$. Note that for $z=w$ \eqref{hr} reduces to \eqref{phi}. So, we present here the proof of Karvounis.
\end{rem}

Furthermore, Karvounis noticed that, using \eqref{hr} and \eqref{phi} we can  relate the invariants $H[R]$ and $H[H]$:

\begin{prop}[K. Karvounis \cite{ka2}] \label{topequivh}
The invariants $H[R]$ and $H[H]$ are topologically equivalent. Specifically, it holds that:
\begin{equation}
H[R](L)(z,w,a,E) = \left( \frac{z}{w} \right)^{n-1} H[H](L)(w,a,\widehat{E}).
\end{equation}
\end{prop}

\begin{rem} \rm \label{depth}
Proposition~\ref{topequivh} implies that the four variables in the original setting of the invariant $H[R]$ can be reduced to three, by setting $z=w$, without any influence on the topological strength of the invariant. However, by taking $z$ and $w$ as  independent variables, we developed the theory in its full generality. We believe that this separation of variables clarifies the logic of the skein theoretic proofs of invariance. In particular, when we keep track of the powers of $z$ in the polynomial, we are looking at the depth of switching operations needed to unlink the link. It may be possible to use this information for further topological invariants of the link.
\end{rem}

\begin{rem} \rm \label{sublinks} 
The combinatorial formula~(\ref{hr}) shows that the strength of  $H[R]$  against $H$ comes from its ability to distinguish certain sublinks of Homflypt-equivalent links. In \cite{chjukala} a list of six 3-component links are given, which are Homflypt equivalent but are distinguished by the invariant $\Theta$ and thus also by $H[R]$.
\end{rem}

\begin{exmp}[L. Kauffman and D. Goundaroulis] \rm \label{ekt} Here is an example showing how $H[R]$ and the combinatorial formula give extra information in the case of two link components.
We will use the ambient isotopy version of the Jones polynomial $V_{K}(q)$ and so first work with a skein calculation of the Jones polynomial, and then with a calculation of the generalized invariant
$V[V](L)(q)$. Recall from Remark~\ref{specializations} that $V[V](L) = \theta(a, E)$ \cite{goula2}. We use the link $ThLink$ first found by Morwen Thisthlethwaite \cite{Th} and generalized by Eliahou, Kauffman and Thistlethwaite \cite{EKT}. This link of two components is not detectable by the Jones polynomial, but it is detectable by our extension of the Jones polynomial. 
 In doing this calculation we (Louis Kauffman and Dimos Goundaroulis) use Dror Bar Natan's Knot Theory package for Mathematica. In this package, the Jones polynomial 
is a function of $q$ and satisfies the skein relation 
$$
q^{-1} V_{K_+}(q) - q V_{K_-}(q) = (q^{1/2} - q^{-1/2}) V_{K_0}(q)
$$ 
where $K_+,K_-,K_0$ is the usual skein triple.
Let 
$$
a = q^2, z=(q^{1/2} - q^{-1/2}),  b = q z, c = q^{-1} z.
$$ 
Then we have the skein expansion formulas:
 $$
V_{K_+} = a V_{K_-} + b V_{K_0} \quad {\mbox \rm and} \quad 
V_{K_-} = a^{-1} V_{K_+} - c V_{K_0}.
$$
 In Figure~\ref{tlink} we show the Thistlethwaite link that is invisible to the Jones polynomial. In the same figure we show an unlink of two components obtained from the Thisthlethwaite link by switching four crossings.
 In Figure~\ref{k1234} we show the links $K_1,K_2,K_3,K_4$ that are intermediate to the skein process for calculating the invariants of $L$ by first switching only crossings between components.
 From this it follows that the knots and links in the figures indicated here satisfy the formula
 $$
V_{ThLink} = bV_{K_1} + abV_{K_2} -ca^2 V_{K_3} - ac V_{K_4} + V_{Unlinked}.
$$
 This can be easily verified by the specific values computed in Mathematica:

 $$
V_{ThLink} = -q^{-1/2} - q^{1/2}
$$
 $$
V_{K_1} = -1+\frac{1}{q^7}-\frac{2}{q^6}+\frac{3}{q^5}-\frac{4}{q^4}+\frac{4}{q^3}-\frac{4}{q^2}+\frac{3}{q}+q
$$
 $$
V_{K_2} =  1-\frac{1}{q^9}+\frac{3}{q^8}-\frac{4}{q^7}+\frac{5}{q^6}-\frac{6}{q^5}+\frac{5}{q^4}-\frac{4}{q^3}+\frac{3}{q^2}-\frac{1}{q}
$$
$$
V_{K_3} = 1-\frac{1}{q^9}+\frac{2}{q^8}-\frac{3}{q^7}+\frac{4}{q^6}-\frac{4}{q^5}+\frac{4}{q^4}-\frac{3}{q^3}+\frac{2}{q^2}-\frac{1}{q}
$$
$$
V_{K_4} = -1-\frac{1}{q^6}+\frac{2}{q^5}-\frac{2}{q^4}+\frac{3}{q^3}-\frac{3}{q^2}+\frac{2}{q}+q
$$
$$
V_{Unlinked}= \frac{1}{q^{13/2}}-\frac{1}{q^{11/2}}-\frac{1}{q^{7/2}}+\frac{1}{q^{3/2}}-\frac{1}{\sqrt{q}}-q^{3/2}
$$

\bigbreak

This is computational proof that the Thistlethwaite link is not detectable by the Jones polynomial.
If we compute $V[V](ThLink)(q)$ then we modify the computation to 
$$
V[V](ThLink)(q) = bV_{K_1} + abV_{K_2} -ca^2 V_{K_3} - ac V_{K_4} + E^{-1}V_{Unlinked}.
$$ 
 and it is quite clear that this is non-trivial when the new variable $E$ is not equal to $1$. 

On the other hand, the Lickorish formula for this case tells us that, for the regular isotopy version of the Jones polynomial $V'[V'](ThLink)(q)$,
$$
V'[V'](ThLink)(q) = \eta(E^{-1} -1)V'_{K_1}V'_{K_2}(q) + V'_{ThLink}(q)
$$
whenever we evaluate a 2-component link. Note that $\eta(E^{-1} -1)$ is non-zero whenever $E \ne 1$.
Thus it is quite clear that the Lickorish formula detects the Thisthlethwaite link since the Jones polyomials of the components of that link are non-trivial.
We have, in this example, given two ways to see how the extended invariant detects the link $ThLink$. The first way shows how the detection works in the extended skein theory.
The second way shows how it works using the Lickorish formula.
\end{exmp}

\begin{figure}
     \begin{center}
     \begin{tabular}{c}
     \includegraphics[width=6.5cm]{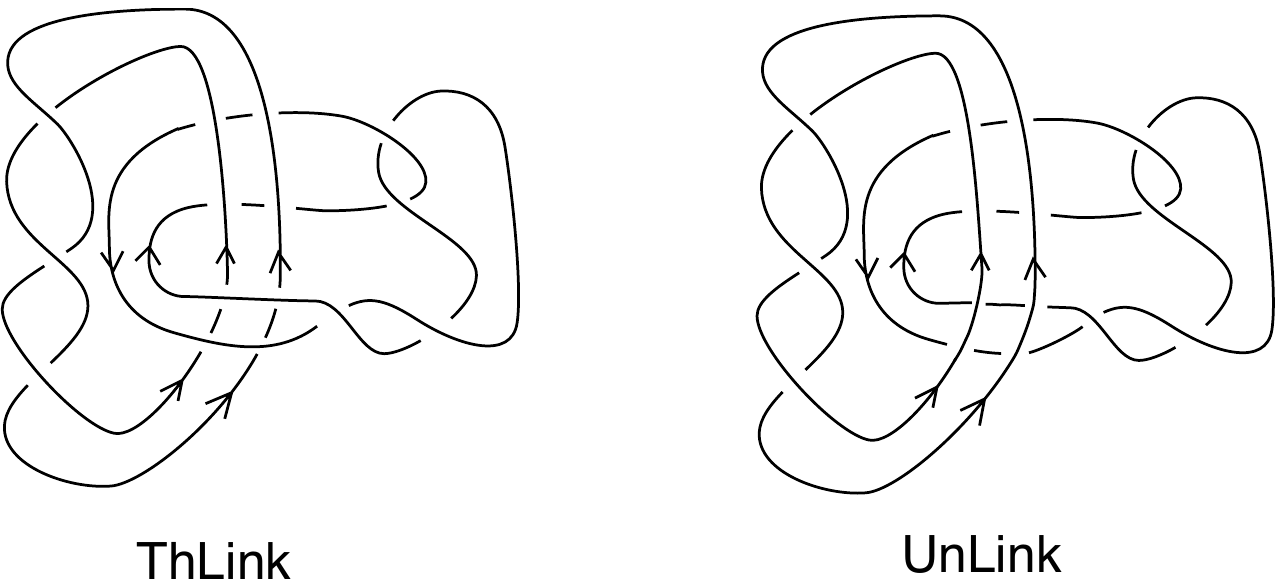}
     \end{tabular}
     \caption{The Thistlethwaite Link and Unlink}
     \label{tlink}
\end{center}
\end{figure}

\begin{figure}
     \begin{center}
     \begin{tabular}{c}
     \includegraphics[width=15cm]{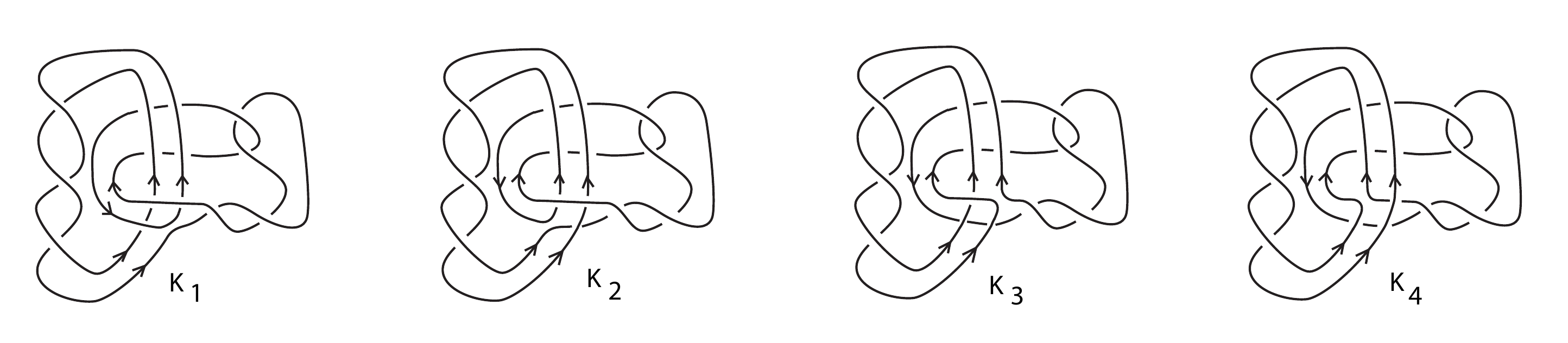}
     \end{tabular}
     \caption{ The links $K_1,K_2,K_3,K_4$}
     \label{k1234}
\end{center}
\end{figure}

\begin{exmp} \rm 
In this example, we point out how to see a generalization of the Jones polynomial (in Kauffman bracket form) as a specialization of our invariant $H[R]$. 
We begin with an expansion formula  for the bracket polynomial that is adapted to our situation. View Figure~\ref{orbracket}. At the top of the figure we show the standard oriented expansion of the bracket.
If the reader is familiar with the usual unoriented expansion \cite{kau5}, then this oriented expansion can be read by forgetting the orientations. The oriented states in this state summation contain smoothings of the type illustrated in the far right hand terms of the two formulas at the top of the figure. We call these {\it disoriented smoothings} since two arrowheads point to each other at these sites. Then by multiplying the two equations by $A$ and by $A^{-1}$ respectively, we obtain a difference formula of the type: 
$ A<K_{+}> - A^{-1} <K_{-}> = (A^{2} - A^{-2})<K_{0}>,$ 
where $K_{+}$ denotes the local appearance of a positive crossing, $K_{-}$ denotes the local appearance of a negative crossing and $K_{0}$ denotes the local appearance of standard oriented smoothing. The difference equation eliminates the disoriented terms. It then follows easily from this difference equation that if we define a {\it curly bracket} by the equation 
\begin{center} 
$ \{K\} = A^{wr(K)} <K> $ 
\end{center} 
where $wr(K)$ is the diagram writhe (the sum of the signs of the crossings of $K$), then we have a Homflypt type relation for $\{K\}$ as follows:
\begin{equation} \label{curlybracket}
\{K_{+}\} -  \{K_{-}\} = (A^{2} - A^{-2})\{K_{0}\}.
\end{equation}
This means that we can regard $\{K\}$ as a specialization of the Homflypt polynomial and so we can use it as the invariant $R$ for $H[R]$. 

\smallbreak
\begin{figure}[H]
\begin{center}
     \includegraphics[width=8cm]{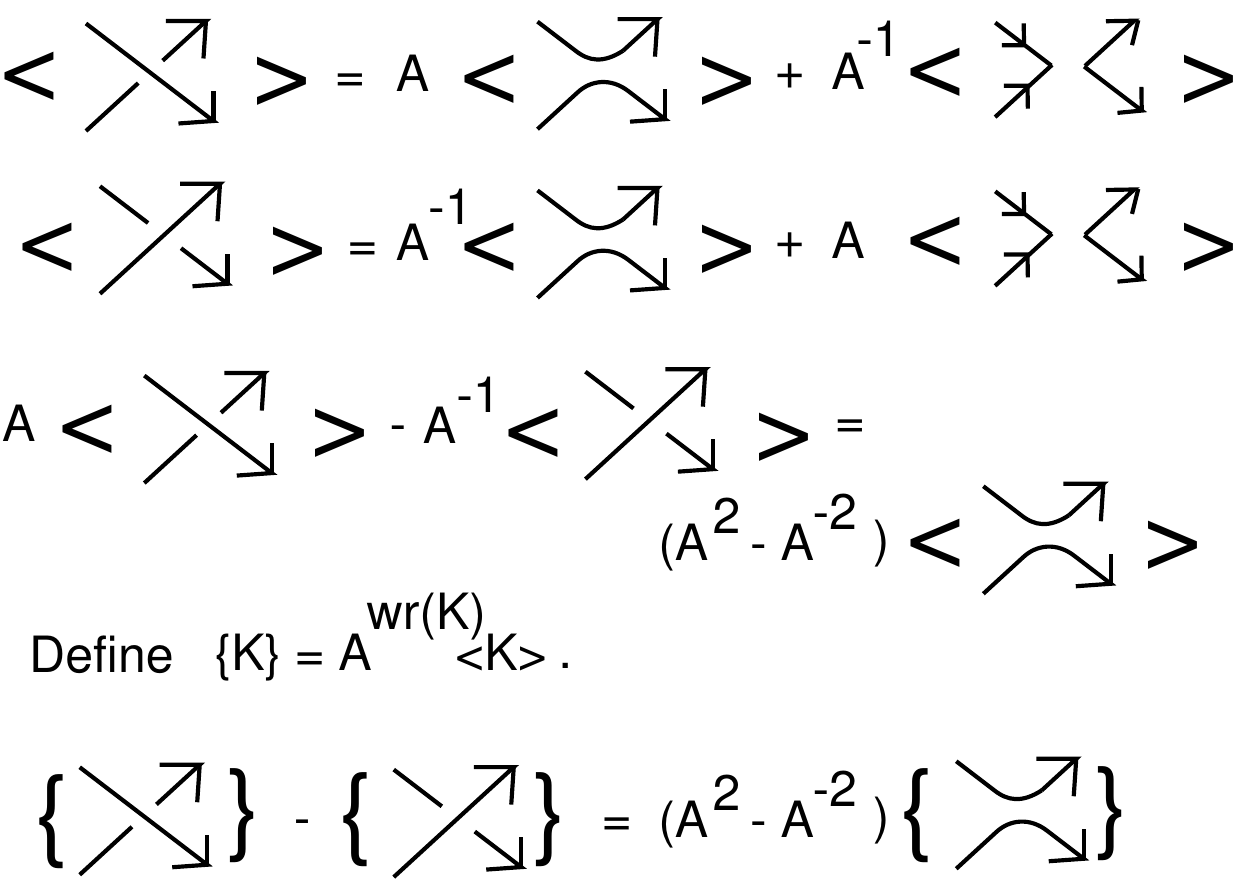}
     \caption{Oriented bracket with Homflypt skein relation}
     \label{orbracket}
\end{center}
\end{figure}

From Figure~\ref{orbracket} it is not difficult to see that \begin{equation} \label{curlyplus} \{K_{+}\} = A^2 \{K_{0}\} + \{K_{\infty}\} \end{equation} and \begin{equation} \label{curlyminus} \{K_{-}\} = A^{-2} \{K_{0}\} + \{K_{\infty}\}.
\end{equation}
Here $K_{\infty}$ denotes the disoriented smoothing shown in the figure. These formulas then define the skein expansion for the curly bracket. The reader should note that the difference of these two expansion equations (\ref{curlyplus}) and (\ref{curlyminus})  is the skein relation (\ref{curlybracket}) for the curly bracket in Homflypt form. 
\end{exmp}

\subsection{A closed combinatorial formula for $D[T]$}\label{secxi}

Recall the rules (T1)-(T5) for the Dubrovnik polynomial and for our extension of it $D[T]$ from Theorem~\ref{doft} and thereafter. We then have the following:

\begin{thm}\label{thmxi}
Let $L$ be an unoriented link with $n$ components and let
\begin{equation} \label{xi}
X[T](L) = (\frac{z}{w})^{n-1}\sum_{k=1}^n \delta^{k-1}\widehat{E_k} \sum_\pi T(\pi L)
\end{equation}
where the second summation is over all partitions $\pi$ of the components of $L$ into $k$ (unordered) subsets and $D(\pi L)$ denotes the product of the Dubrovnik polynomials of the $k$ sublinks of $L$ defined by $\pi.$
Furthermore, $\widehat{E}_k = (\widehat{E}^{-1} - 1)(\widehat{E}^{-1} - 2) \cdots (\widehat{E}^{-1} - k + 1)$, with $\widehat{E} = E \frac{z}{w}$, $\widehat{E}_1 =1$,  and $\delta = \frac{a - a^{-1}}{w} + 1$.  If $D[T]$ satisfies  properties (1) and (2) of Theorem~\ref{doft} then $D[T](L) = X[T](L)$. 
  Independently of that, $X[T]$ is a well-defined regular isotopy invariant of oriented links and $X[T]$ satisfies the relations (1) and (2) of Theorem~\ref{doft}.  Therefore $X[T]$ is a closed formula for $D[T]$.   
\end{thm}

\begin{proof}
We will first show that if $Q[T]$ satisfies  properties (1) and (2) of Theorem~\ref{doft} then $Q[T]$ can be given by the formula (\ref{xi}). Before proving this, note the following equalities:
\begin{align*}
T(L_1 \sqcup L_2) &= \delta \, T(L_1) T(L_2),
\\
D[T](L_1 \sqcup L_2) &= \frac{\delta}{E} \, D[T](L_1) D[T](L_2).
\end{align*}

Suppose that a diagram of $L$ is given. The proof is by induction on $n$ and on the number, $u$, of crossing changes between distinct components required to change $L$ to $n$ unlinked knots. If $n=1$ there is nothing to prove. So assume the result true for $n-1$ components and $u-1$ crossing changes and prove it true for $n$ and $u$.

The induction starts when $u = 0$. Then $L$ is the  union of $n$ unlinked components $L_1, L_2, \dots, L_n$. A classic elementary result concerning the Dubrovnik polynomial shows that 
$$
T(L) = \delta^{n-1}T(L_1)T(L_2)\cdots T(L_n).
$$ 
Furthermore, in this situation, for any $k$ and $\pi$, $T(\pi L) = \delta^{n-k}T(L_1)T(L_2)\cdots T(L_n)$. So it is required to prove that

\begin{equation}\label{dunlink}
\delta^{n-1} E^{1-n} = \left( \frac{z}{w} \right)^{n-1} \delta^{n-1} \sum_{k=1}^n S(n,k)(\widehat{E}^{-1} - 1)(\widehat{E}^{-1} - 2) \cdots (\widehat{E}^{-1} - k + 1),
\end{equation}
where $S(n,k)$ is the number of partitions of a set of $n$ elements into $k$ subsets. Now it remains to prove that:
\begin{equation}\label{dstirling}
E^{1-n} \left( \frac{z}{w} \right)^{1-n} = \widehat{E}^{1-n} = \sum_{k=1}^n S(n,k)(\widehat{E}^{-1} - 1)(\widehat{E}^{-1} - 2) \cdots (\widehat{E}^{-1} - k + 1).
\end{equation}
However, in the theory of combinatorics, $S(n,k)$ is known as a Stirling number of the second kind and this required formula is a well known result about such numbers.

Now suppose that $u > 0$. Suppose that in a sequence of $u$ crossing changes that changes $L$, as above, into unlinked knots, the first change is to a crossing $c$ between components $L_1$ and $L_2$ with relative sign $\epsilon$. Let $L^\prime$ be $L$ with the crossing changed and $L^0$ and $L^{\infty}$ be $L$ with the crossing annulled in the two possible ways. Now, from the definition of $D[T]$,
$$
D[T](L) =  D[T](L^\prime) + \epsilon z\big(D[T](L^0) - D[T](L^{\infty})\big)
$$

\noindent The induction hypotheses imply that the result is already proved for $L^\prime$, $L^0$ and $L^{\infty}$ so:
\begin{equation}\label{dstar}
D[T](L)=  \sum_{k=1}^n \delta^{k-1}\widehat{E_k} \sum_{\pi^\prime} T(\pi^\prime L^\prime) + \epsilon z (\sum_{k=1}^{n-1} \delta^{k-1}\widehat{E_k} \sum_{\pi^0} T(\pi^0 L^0) - \sum_{k=1}^{n-1} \delta^{k-1}\widehat{E_k} \sum_{\pi^{\infty}} T(\pi^{\infty} L^{\infty}))
\end{equation}
where $\pi^\prime$ runs through the partitions of the components of $L^\prime$, $\pi^0$ through those of $L^0$ and $\pi^{\infty}$ through those of $L^{\infty}$.

A sublink $X^0$ of $L^0$ can be regarded as a sublink $X$ of $L$ containing $L_1$ and $L_2$ but with $L_1$ and $L_2$ fused together by annulling the crossing at $c$. Similarly,  a sublink $X^{\infty}$ of $L^{\infty}$ can be regarded as a sublink $X$ of $L$ containing $L_1$ and $L_2$ but with $L_1$ and $L_2$ fused together by annulling the crossing at $c$.
Let $X^\prime$ be the sublink of $L^\prime$ obtained from $X$ by changing the crossing at $c$. Then
$$
T(X) =  T(X^\prime) +\epsilon w( T(X^0) - T(X^{\infty})).
$$
This means that the second (big) term in (\ref{dstar}) is
\begin{equation}\label{dbig_term}
\sum_{k=1}^{n-1} \delta^{k-1}\widehat{E_k} \sum_{\rho} \bigl( T(\rho L) - T(\rho^\prime L^\prime  ) \bigr),
\end{equation}
where the summation is over all partitions $\rho$ of the components of $L$ for which $L_1$ and $L_2$ are in the same subset and $\rho^\prime$ is the corresponding partition of the components of $L^\prime$.

\noindent Thus, substituting (\ref{dbig_term}) in (\ref{dstar}) we obtain:
\begin{equation}\label{dalmostdone}
D[T](L)=  \sum_{k=1}^n \delta^{k-1}\widehat{E_k} \biggl( \sum_{\pi^\prime}  T(\pi^\prime L^\prime) + \sum_{\rho} \bigl( T(\rho L) - T(\rho^\prime L^\prime)\bigr) \biggr),
\end{equation}
where $\pi^\prime$ runs through all partitions of $L^\prime$ and $\rho$ through partitions of $L$ for which $L_1$ and $L_2$ are in the same subset. Note that, for $k=n$ the second sum is zero. 

\noindent Therefore
\begin{equation}\label{ddone}
D[T] = \sum_{k=1}^n \delta^{k-1}\widehat{E_k} \biggl( \sum_{\pi^\prime} T(\pi^\prime L^\prime) + \sum_{\rho}  T(\rho L) \biggr),
\end{equation}
where  $\pi^\prime$  runs through only partitions of $L^\prime$ for which $L_1$ and $L_2$ are in different subsets and $\rho$ through all partitions of $L$ for which $L_1$ and $L_2$ are in the same subset. 

Note that, for any partition $\pi$ of the components of $L$ inducing partition $\pi^\prime$ of $L^\prime$, if $L_1$ and $L_2$ are in the same subset then 
we can have a difference between $D(\pi L)$ and $D(\pi^\prime L^\prime)$, but when $L_1$ and $L_2$ are in different subsets then
 \begin{equation}\label{ddiff_subsets}
T(\pi^\prime L^\prime) = T(\pi L). 
\end{equation}
 
\noindent Hence, using (\ref{ddone}) and also (\ref{ddiff_subsets}), we obtain:
$$
D[T] = \sum_{k=1}^n \delta^{k-1}\widehat{E_k} \sum_\pi T(\pi L)
$$
and the induction is complete.

Now we observe that the formula for $X[T](L)$ satisfies properties (1) and (2) of Theorem~\ref{doft}. The proof of this is actually contained in the formalism of the proof that we have already made above. The reader will see that the argument from (\ref{dstar}) to the end of the proof can be read as a verification that the formula $X[T](L)$ satisfies the skein relation (1)  of Theorem~\ref{doft} and the argument at the beginning of the proof involving formulas (\ref{dunlink}) and (\ref{dstirling}) can be regarded as a verification of property (2)  of Theorem~\ref{doft} for the invariant $X[T]$. The proof is now complete.
\end{proof}

\begin{rem} \rm 
Note that the formula (\ref{xi}) can be regarded by itself as a definition of the invariant $D[T]$, since the right-hand side of the formula is an invariant of regular isotopy, as $D$ is invariant of regular isotopy. 
\end{rem}

Furthermore,  Remark~\ref{sublinks} 
 applies also for the invariant $D[T]$. Finally, Remark~\ref{fiandh} applies here too, so, adapting  Proposition~\ref{topequivh}, we have in analogy:

\begin{prop}\label{topequivd}
The invariants $D[T]$ and $D[D]$ are topologically equivalent. Specifically, it holds that:
\begin{equation}
D[T](L)(z,w,a,E) = \left( \frac{z}{w} \right)^{n-1} D[D](L)(w,a,\widehat{E}).
\end{equation}
\end{prop}

Proposition~\ref{topequivd} implies that the four variables in the original setting of the invariant $D[T]$ can be reduced to three without any influence on the topological strength of the invariant. However, we opted for keeping $z$ and $w$ as  independent variables, in order to develop the theory in its full generality. Recall also Remark~\ref{depth}, which applies for $D[T]$  as well.

\subsection{A closed combinatorial formula for $K[Q]$}\label{secpsi}

Recall the basic skein formulas (Q1)--(Q5) for the regular isotopy 
invariant $Q$, used here  as a fully independent version of the Kauffman polynomial, and the rules for our generalization $K[Q]$ of Theorem~\ref{kofq}. Then we have the following:

\begin{thm} \label{thmpsi}
Let $L$ be an unoriented link with $n$ components and let 
\begin{equation} \label{psi}
\Psi[Q] (L) = i^{wr(L)}  (\frac{z}{w})^{n-1} \sum_{k=1}^n \gamma^{k-1}\widehat{E_k} \sum_\pi i^{-wr(\pi L)}Q(\pi L). 
\end{equation}
where the second summation is over all partitions $\pi$ of the components of $L$ into $k$ (unordered) subsets and $K(\pi L)$ denotes the product of the Kauffman polynomials of the $k$ sublinks of $L$ defined by $\pi$. 
The term $wr(\pi L)$ denotes the sum of the writhes of the parts of the partitioned link $\pi L$.
Furthermore, $\widehat{E_k} = (\widehat{E}^{-1} - 1)(\widehat{E}^{-1} - 2) \cdots (\widehat{E}^{-1} - k + 1)$, with $\widehat{E_1} =1$, $\widehat{E} = \frac{z}{w}E$ and 
$\gamma = \frac{a + a^{-1}}{w} -1$. Then $\Psi[Q]$ is a well-defined regular isotopy invariant of oriented links and $\Psi[Q]$ coincides with $K[Q]$.
\end{thm}

\begin{proof}
In order to prove this Theorem, we first discuss a translation between the Kauffman and Dubrovnik polynomials. We then use this translation to deduce a combinatorial formula for $K[Q]$ from the combinatorial formula we have already proved for  the Dubrovnik polynomial extension $D[T]$. The following equation is the translation formula from the  Dubrovnik to Kauffman polynomial, observed by W.B.R. Lickorish \cite{kau4}:
$$
D(L)(a,z) = (-1)^{c(L)+1} \, i^{-wr(L)}K(L)(ia,-iz).
$$ 
Here, $c(L)$ denotes the number of components of $L$, $i^2 = -1$, and $wr(L)$ is the writhe of $L$ for some choice of orientation of $L$. The translation formula is independent of the particular choice of orientation for $L$. By the same token, we have the following formula translating the Kauffman polynomial to the Dubrovnik polynomial.
$$
K(L)(a,z) = (-1)^{c(L)+1} \,  i^{wr(L)} D(L)(-ia, iz).
$$
These formulas are proved by checking them on basic loop values and then using induction via the skein formulas for the two polynomials. This same method of proof shows that the same translation 
occurs between our generalizations of the Kauffman polynomial $K[Q]$  and the Dubrovnik polynomial $D[T]$.  
In particular, we have
\begin{equation} \label{psitoxi}
D[T](L)(a,z,w) = (-1)^{c(L)+1} \, i^{-wr(L)}K[Q](L)(ia,-iz,-iw)
\end{equation} 
and
\begin{equation} \label{xitopsi}
K[Q](L)(a,z,w) = (-1)^{c(L)+1} \, i^{wr(L)}D[T](L)(-ia,iz,iw).
\end{equation}

 We will show that $K[Q]$ of Theorem~\ref{kofq} can be given by the formula (\ref{psi}). We know that
$$
D[T](L) = (\frac{z}{w})^{n-1}\sum_{k=1}^n \delta^{k-1}\widehat{E_k} \sum_\pi T(\pi L),
$$
and 
$$
T(\pi L)(a,w) = (-1)^{c(\pi L)} \, i^{-wr(\pi L)}Q(\pi L)(ia,-iw).
$$ 
Here it is understood that $wr(\pi L)$ is the sum of the writhes of the parts of the partition of $L$ corresponding to $\pi$.
Note that $T(\pi L)(a,w)$ is a product of the Dubrovnik evaluations of the parts of the partition $\pi L$.
The term $c(\pi L)$ is equal to the sum 
$$
c(\pi L) = \sum_{\sigma} (c(\sigma) + 1) = \sum_{\sigma} c(\sigma)  + k = c(L) + k
$$  
where $\sigma$ runs over the parts of the partition $\pi L$.
Here $k$ is the number of parts in $\pi L$.
Thus
$$
D[T](L)(a,z,w) = (\frac{z}{w})^{n-1} \sum_{k=1}^n \delta(a,w)^{k-1}\widehat{E_k} \sum_\pi T(\pi L)(a,w)
$$
\begin{center}
$
=  (\frac{z}{w})^{n-1}  \sum_{k=1}^n \delta(a,w)^{k-1}\widehat{E_k} \sum_\pi (-1)^{c(\pi L)} \, i^{-wr(\pi L)}Q(\pi L)(ia,-iw).
$
\end{center}
We also know that 
\begin{center}
$
K[Q] (L)(a,z,w) = (-1)^{c(L)+1} \, i^{wr(L)}D[T](L)(-ia,iz,iw),
$
\end{center} 
and so we have
$$K[Q] (L)(a,z,w) = $$
$$(-1)^{c(L)+1} \, i^{wr(L)}  (\frac{z}{w})^{n-1} \sum_{k=1}^n \delta(-ia,iw)^{k-1}\widehat{E_k} \sum_\pi (-1)^{c(\pi L)} \, i^{-wr(\pi L)}Q(\pi L)(a,w).$$
Now we have that 
\begin{center}
$ 
\delta(-ia,iw) = \left((-ia) - (-ia)^{-1}\right)/(iw) + 1 = -\left( (a+a^{-1})/w - 1\right) = - \gamma(a,w).
$
\end{center}  
Therefore
$$K[Q] (L)(a,z,w) =$$
$$ (-1)^{c(L)+1} \, i^{wr(L)}  (\frac{z}{w})^{n-1} \sum_{k=1}^n (-1)^{k-1}\gamma(a,w)^{k-1}\widehat{E_k} \sum_\pi (-1)^{c(L) + k} \, i^{-wr(\pi L)}Q(\pi L)(a,w).$$
Thus
\begin{center}
$
K[Q] (L)(a,z,w) = i^{wr(L)}  (\frac{z}{w})^{n-1} \sum_{k=1}^n \gamma(a,w)^{k-1}\widehat{E_k} \sum_\pi i^{-wr(\pi L)}Q(\pi L)(a,w).
$
\end{center}
Hence
\begin{center}
$
K[Q] (L) = i^{wr(L)}  (\frac{z}{w})^{n-1} \sum_{k=1}^n \gamma^{k-1}\widehat{E_k} \sum_\pi i^{-wr(\pi L)}Q(\pi L).
$
\end{center} 
This completes the proof.
 \end{proof}

Note that the formula (\ref{psi}) can be regarded by itself as a definition of the invariant $K[Q]$, since the right-hand side of the formula is an invariant of regular isotopy, since $D$ is invariant of regular isotopy. Furthermore, Remark~\ref{sublinks} applies also for the invariant $K[Q]$. The same for Remark~\ref{fiandh}. So, adapting  Proposition~\ref{topequivh}, we have in analogy:

\begin{prop}\label{topequivk}
The invariants $K[Q]$ and $K[K]$ are topologically equivalent. Specifically, it holds that:
\begin{equation}
K[Q](L)(z,w,a,E) = \left( \frac{z}{w} \right)^{n-1} K[K](L)(w,a,\widehat{E}).
\end{equation}
\end{prop}

Proposition~\ref{topequivk} implies that the four variables in the original setting of the invariant $K[Q]$ can be reduced to three without any influence on the topological strength of the invariant. Recall also Remark~\ref{depth}, which is valid for $K[Q]$ too.

\begin{rem} \rm
As noted in the Introduction, in Theorems~\ref{doft} and~\ref{kofq} the basic invariants $T(w,a)$ and $Q(w,a)$ could be replaced by specializations of the Dubrovnik and the Kauffman polynomial respectively and, then, the invariants $D[T]$ and $K[Q]$ can be regarded as generalizations of these specialized polynomials.  For example, if $a=1$ then $Q(w,1)$ is the Brandt--Lickorish--Millett--Ho polynomial and if $w= A+A^{-1}$ and $a= -A^3$ then $Q ( A+A^{-1}, -A^3)$ is the Kauffman bracket polynomial. In both cases the invariant $K[Q]$ generalizes these polynomials.
\end{rem}

\section{Conclusions}\label{conclusions}

In this paper we first gave a direct skein-theoretic proof of the existence and uniqueness of the generalization of the regular isotopy version of the Homflypt polynomial and its ambient isotopy counterpart, the new invariant $\Theta (q,\lambda,E)$ \cite{chjukala}. We then generalized to new skein link invariants the Dubrovnik and the Kauffman polynomials and we provided closed defining combinatorial formulae for all these new invariants.

\section{Discussing mathematical directions and applications}\label{directions}

We shall now discuss some possible research directions emanating from our results.

\begin{enumerate}
\item  Computer calculations of the new skein link invariants need to be done in order to check their topological strength.

\item  Just as the invariant $\Theta(q,\lambda,E) = P[P]$ is related to the Yokonuma--Hecke algebra and to the algebra of braids and ties \cite{AJ1}, it would be interesting to see the invariants $D[D]$ and $K[K]$, or rather $Y[Y]$ and $F[F]$ defined in (\ref{yzfromdt}) and (\ref{fsfromkq}), related to some knot algebras, such as the framization of the BMW algebra proposed in \cite{jula4,jula5}. See also \cite{AJ3}.

\item The categorification of the new skein invariants is another interesting problem and, for the invariant $\theta(q,E)$, which generalizes the Jones polynomial and is a specialization of $\Theta(q,z,E)$ \cite{goula2}, it is the object of the ongoing research project \cite{cgkl}.

\item The invariants in the present paper can be formulated in terms of the invariants of tied links, using the methods of Aicardi and Juyumaya \cite{AJ2,AJ3}. This will be investigated in a subsequent paper. 

\item For a given link diagram $L$ and choice of skeining $S$ we arrive at collection of stacks of knot diagrams. Let $N(L,S)$ denote the total number of knot diagrams running over all the stacks for this choice of skeining. Let $N(L,S)$  denote the minimum of $N(L,S)$  over all choices of skeining for the diagram L. Let $N[L]$ denote the minimum of $N(L,S)$  over all diagrams regularly isotopic to $L$ ( or ambient isotopic if we want). Then $N(L,S)$  is an invariant (of regular isotopy) of $L$ associated with this skein process. We would like to know more about this invariant. Another interesting feature is the least number of mixed crossing switches occuring in a given skein tree, minimized over all diagrams. This is also an invariant of the link.

\item  One avenue for further research is to extend the approach of the present paper to the area of skein modules and invariants of links in three-manifolds, and invariants of three-manifolds (see for example \cite{Tu,P,HK,La1,La2,GM,DL3,chpa2,fljula}). 

\end{enumerate} 

We further contemplate how these new ideas can be applied to physical situations.
We present these indications of possible applications here with the full intent to pursue them in subsequent publications.

\begin{enumerate}
\item   
In \cite{kaula2} we give state summation models for the new invariants, based on the skein process for computing them. That is, we use the skein process that first unlinks links and then computes invariants on stacks of knots. We systematize this process and write it as a state summation that is a version of the skein template algorithm explained in \cite{kau6}. The interest in rewriting as a state summation is that we can then interface our work with ideas from both statistical mechanics and from state models for knot invariants. We also can examine how this state summation process with its multiple levels may be analogous to the way certain physical systems have structural levels. These matters are discussed in detail in \cite{kaula2}.
We will return to these themes in subsequent papers on the subject. 

\item Reconnection (in vortices). In a knotted vortex in a fluid or plasma (for example in solar flares)  \cite{Irv} one has a cascade of changes in the vortex topology as strands of the vortex undergo reconnection (see \cite{RL}). The process goes on until the vortex has degenerated into a disjoint union of unknotted simpler vortices. This cascade or hierarchy of interactions is reminiscent of the way the skein template algorithm proceeds to produce unlinks.
Studying reconnection in vortices may be facilitated by making a statistical mechanics summation related to the cascade. Such a summation will be analogous the state summations we have described here.

\item In DNA, strand switching using topoisomerase of types I and II is vital for the structure of DNA recombination and DNA replication \cite{Sumners}. The mixed interaction of topological change and physical evolution of the molecules in vitro may benefit from a mixed state summation that averages quantities respecting the hierarchy of interactions.

\item Remarkably, the process of separation and evaluation that we have described here is analogous to proposed processing of Kinetoplast DNA \cite{Kinetoplast} where there are huge links of DNA circles and these must undergo processes that both unlink them from one another and produce new copies for each circle of DNA. The double-tiered structure of DNA replication for the Kinetoplast appears to be related
to the mathematical patterns of our double state summations. For chainmail DNA, if the reader examines the Wiki on Kinetoplast DNA, he/she will 
note that  Topoisomerase II figures crucially in the self-replication \cite{Kineto}.

\item We would like to know whether we could have physical situations that would have the kind of a mixture that is implicit in the state summation for $H[R]$, where the initial skein template state sum yields a sum over
$R$-evaluations, and $R$ may itself have a state summation structure \cite{kaula2}. One possible example in the physical world is a normal statistical mechanical situation, where one can have multiple types of materials, all present together, each having different energetic properties. This can lead to a mixed partition function, possibly not quite ordered 
in the fashion of our algorithm. This would involve a physical hierarchy of interactions so that there would be a double (or multiple)  tier resulting from that hierarchy.

\item Mixed state models can occur in physical situations when we work with systems of systems. There are many examples of this multiple-tier situation in systems physical and biological. We look for situations where a double state sum would yield new information. For example, in a quantum Hall system \cite{Haldane}, the state of the system is in its quasi-particles, but each quasi-particle is itself a vortex of electrons related to a magnetic field line. So the quasi-particles are themselves localized physical systems. Some of this is summarized in the Laughlin wave function for quantum Hall \cite{Haldane}. Not a simple situation, but a very significant one. There should be other important examples.
\end{enumerate}

\end{document}